\documentclass[english,12pt]{amsart}

\usepackage[latin9]{inputenc}
\usepackage{amsmath,multirow}
\usepackage{amsthm}
\usepackage{amssymb}
\usepackage{url} 
\usepackage{color} 
\usepackage{caption}
\usepackage{array}
\usepackage{amsopn}
\usepackage{mathtools}
\usepackage{bm,color}

\usepackage{algorithm}          
\usepackage{algpseudocode}      
\algblock{Input}{EndInput}
\algnotext{EndInput}
\algblock{Output}{EndOutput}
\algnotext{EndOutput}
\algblock{Compute}{EndCompute}
\algnotext{EndCompute}
\algblock{Initialize}{EndInitialize}
\algnotext{EndInitialize}
\algblock{Procedure}{EndProcedure}
\algnotext{EndProcedure}
\algblockdefx[Foreach]{Foreach}{EndForeach}[1]{\textbf{for each} #1 \textbf{do}}{\textbf{end for}}

\makeatletter
 
\theoremstyle{plain}
\newtheorem{thm}{\protect\theoremname}[section]
  \newtheorem{proposition}[thm]{\protect\propositionname}
  
  \newtheorem{lem}[thm]{\protect\lemmaname}
  \newtheorem{alg}[thm]{Algorithm}

  \theoremstyle{plain}
  
  \theoremstyle{plain}
  
  \theoremstyle{remark}
  \newtheorem{rem}[thm]{\protect\remarkname}
  \theoremstyle{definition}
  \newtheorem{example}[thm]{\protect\examplename}
  
  \theoremstyle{definition}

\newcommand{\PP}{\mathbb{P}}
\newcommand{\RR}{\mathbb{R}}
\newcommand{\CC}{\mathbb{C}}
\newcommand{\ZZ}{\mathbb{Z}}
\newcommand{\NN}{\mathbb{N}}
\newcommand{\calV}{\mathcal{V}}  

\DeclareMathOperator{\Mat}{Mat}

\newcommand{\ndot}{{n_\bullet}}   
\newcommand{\npdot}{{n_\bullet'}}   
\newcommand{\nppdot}{{n_\bullet''}}   

\newcommand{\blambda}{\boldsymbol{\lambda}}
\newcommand{\btheta}{\boldsymbol{\theta}}
\newcommand{\bfe}{\mathbf{e}}
\newcommand{\bfx}{\mathbf{x}}
\newcommand{\bfy}{\mathbf{y}}
\newcommand{\bfu}{\mathbf{u}}
\newcommand{\bfz}{\mathbf{z}}

\DeclareMathOperator{\aff}{aff}
\DeclareMathOperator{\rank}{rank}

\newcommand{\bL}{\mathbf{L}}  
\newcommand{\bfT}{\mathbf{T}}
\newcommand{\bM}{\mathbf{M}}  

\DeclareMathOperator{\Dim}{Dim} 
\DeclareMathOperator{\Deg}{Deg} 

\newcommand{\ol}[1]{\overline{#1}}

\newcommand{\defcolor}[1]{{\color{blue}#1}}
\newcommand{\demph}[1]{\defcolor{{\sl #1}}}
  
\numberwithin{equation}{section} 

\makeatother

\usepackage{babel}
  \providecommand{\corollaryname}{Corollary}
  \providecommand{\definitionname}{Definition}
  \providecommand{\examplename}{Example}
  \providecommand{\notationname}{Notation}
  \providecommand{\lemmaname}{Lemma}
  \providecommand{\propositionname}{Proposition}
  \providecommand{\questionname}{Question}
  \providecommand{\problemname}{Problem}
  \providecommand{\remarkname}{Remark}
\providecommand{\theoremname}{Theorem}
\newcounter{FNC}[page]
\def\fauxfootnote#1{{\addtocounter{FNC}{2}${\color{magenta}^\fnsymbol{FNC}}$%
     \let\thefootnote\relax\footnotetext{{\color{magenta}$^\fnsymbol{FNC}$#1}}}}

\begin{document}

\title{A numerical toolkit for multiprojective varieties}
\author[J.~D.~Hauenstein]{Jonathan D.~Hauenstein}
\address{Department of Applied \& Computational Mathematics \& Statistics\\
         University of Notre Dame\\
         Notre Dame, IN  46556\\         
         USA}
\email{hauenstein@nd.edu}
\urladdr{{http://www.nd.edu/~jhauenst}}
\author[A.~Leykin]{Anton Leykin} 
\address{Anton Leykin\\
         School of Mathematics\\ 
         Georgia Institute of Technology\\ 
         686 Cherry Street\\ 
         Atlanta, GA 30332-0160 USA\\ 
         USA} 
\email{leykin@math.gatech.edu} 
\urladdr{http://people.math.gatech.edu/~aleykin3} 
\author[J.~I.~Rodriguez]{Jose Israel Rodriguez}
\address{Department of Mathematics\\
         University of Wisconsin\\
         Madison, WI 53706\\         
         USA}
\email{Jose@math.wisc.edu}
\urladdr{{http://www.math.wisc.edu/~jose/}}
\author[F.~Sottile]{Frank Sottile}
\address{Department of Mathematics\\
         Texas A\&M University\\
         College Station, TX 77843\\
         USA}
\email{sottile@math.tamu.edu}
\urladdr{{http://www.math.tamu.edu/~sottile/}}

\thanks{Research of Hauenstein supported in part by NSF grant CCF-1812746} 
\thanks{Research of Leykin supported in part by NSF grant DMS-1151297} 
\thanks{Research of Rodriguez supported in part by NSF grant DMS-1402545}         
\thanks{Research of Sottile supported in part by NSF grant DMS-1501370}
%
%

\begin{abstract}
 A numerical description of an algebraic subvariety of projective space is given by a general linear section, called a
 witness set. 
 For a subvariety of a product of projective spaces (a multiprojective variety), the corresponding numerical description
 is given by a witness collection, whose structure is more involved.
 We build on recent work to develop a toolkit for the numerical manipulation of multiprojective varieties that operates on
 witness collections and use this toolkit in an algorithm for numerical irreducible decomposition of multiprojective
 varieties.
 The toolkit and decomposition algorithm are illustrated throughout 
 in a series of examples.
\end{abstract}

\maketitle

\vspace{-14pt}
\section*{Introduction}\label{s:intro}
Numerical algebraic geometry~\cite{SW05} uses numerical analysis to manipulate and study algebraic varieties on a
computer. 
In numerical algebraic geometry, a subvariety $X$ of affine or projective space is represented by a witness set, which
includes a finite set of points in a general linear section of $X$~\cite{SV}.
Algorithms to manipulate a variety operate on its witness sets.
A fundamental algorithm is numerical irreducible decomposition~\cite{SVW_decomposition}, which uses
monodromy~\cite{SVW_monodromy} and a trace test~\cite{SVW_trace} to partition a witness set of a reducible variety into
witness sets for each irreducible component.

Oftentimes, a variety possesses additional structure, such as multihomogeneity, which is when its defining polynomials
are separately homogeneous in disjoint subsets of variables.
For example, the determinant $\det(x_{i,j})$ is separately linear in the variables of each column.
Such a variety is naturally a subvariety of a product of projective spaces (a multiprojective variety).
For the $n\times n$ determinant, this product is $\PP^{n-1}\times\dotsb\times\PP^{n-1}$ ($n$ factors).
We seek algorithms for multiprojective varieties that are adapted to their structure.

Algorithms for numerically solving systems of multihomogeneous polynomials are classical~\cite{MS87}.
A useful notion of witness set---a witness collection---for multiprojective varieties, along with fundamental 
algorithms, was given in~\cite{HR15}.
There, it was observed that the trace test could not be applied naively to a witness collection.
Consequently, for numerical irreducible decomposition, a witness collection for a multiprojective variety must be
transformed into a witness set for a projective or affine variety.
Since a multiprojective variety is a projective variety under the Segre embedding, that could be used for numerical
irreducible decomposition.
In general, the Segre embedding dramatically increases the ambient dimension and degree
as we show in Section~\ref{S:I.7}. 
Passing instead to an affine patch in the product of projective spaces preserves the ambient dimension, but
we do not know an algorithm of acceptable complexity to compute a witness set from a witness collection unless the variety
is a curve. 
This is because for curves the Segre embedding can be used without increasing the
  degree so much as seen in Section~\ref{sec:coarse}.

A version of numerical irreducible decomposition was proposed for subvarieties in the product of two projective
spaces~\cite{traceLRS}.
This reduces numerical irreducible decomposition to  that of a curve,
decreasing the size of the witness sets.
We extend that analysis to arbitrary multiprojective varieties.
We present four geometric constructions and corresponding algorithms that operate on witness collections, and together
provide a toolkit for the numerical manipulation of multiprojective varieties.
A key ingredient is the support of a multiprojective variety~\cite{CLZ}, which is a multiprojective
version of dimension. 

The computation of this (multi)dimension locally at a point reduces to linear algebra.
When the multidimension decomposes as a product, the corresponding variety is also a product as is a witness collection for
it. 
We next explain how witness collections transform under birational maps that change the multiprojective structure, and
finally how a witness collection behaves under slicing with a hyperplane.
We also give an algorithm based on monodromy for computing a witness collection.
The utility of this toolkit is illustrated in an algorithm for numerical irreducible decomposition of multiprojective
varieties.
We use these tools to reduce the numerical irreducible decomposition to that of a curve in affine space, to which we may
apply an efficient trace test.
This generalizes the method of~\cite{traceLRS}, from two to arbitrarily many projective factors.

Algorithms in numerical algebraic geometry typically operate on affine varieties.
A subvariety $X\subset\PP^n$ of projective space is replaced by its intersection \defcolor{$X_{\aff}$} with a general
affine patch $\CC^n\subset\PP^n$ where a general linear polynomial $\ell$ does not vanish.
The same approach could be followed for a multiprojective variety $X\subset\PP^{n_1}\times\dotsb\times\PP^{n_k}$ by taking
affine patches in each projective factor and combining them, giving $X_{\aff}\subset\CC^{n_1+\dotsb+n_k}$.
This neglects the given structure and increases the size and complexity of the witness set, which is
particularly significant when $X$ is neither a curve nor a hypersurface.
We will work with multiaffine varieties $X_{\aff}\subset\CC^{n_1}\times\dotsb\times\CC^{n_k}$, using algorithms that respect
this decomposition and are compatible with the multiprojective structure of $X$.

This paper is structured as follows. 
In Section~\ref{sec:wc}, we define witness collections of multiprojective and multiaffine varieties, and introduce some
running examples.
In Section~\ref{sec:ldt}, we give an algorithm to compute (multi)dimension locally and an algorithm based on monodromy to compute a witness collection.
In Section~\ref{sec:product}, we show how to detect and exploit that a variety is a product.
In Section~\ref{sec:coarse}, we show how to transform witness collections under the birational maps that correspond to
changing the multiprojective structure of a variety.
In Section~\ref{sec:slice}, we show how to perform a dimension reduction based on intersections with linear spaces that
preserves (ir)reducibility.
In Section~\ref{S:NID}, we sketch an algorithm for numerical irreducible decomposition that uses this toolkit.
In Section~\ref{sec:FiberProductEx}, we consider two examples based on fiber products which naturally yield multihomogeneous systems.
Apart from showcasing our toolkit, these examples demonstrate that using multiprojective structure leads to significant reduction in the size of computations.

%
\section{Background}\label{sec:wc}
For a finite set $F$ of polynomials, let \defcolor{$\calV(F)$} be its variety, the subset of affine or projective
space (or products thereof) where every polynomial in $F$ vanishes.
A variety will also be a union of irreducible components of such a set $\calV(F)$.
We first recall numerical homotopy continuation, then witness sets~\cite[Ch.~13]{SW05}, multiprojective
varieties,  witness collections~\cite{HR15}, and multiaffine varieties.
We end by introducing our running examples.

\subsection{Numerical homotopy continuation}

Many algorithms described herein are based on numerical homotopy continuation.
A \demph{homotopy} is a system of polynomials $H(\bfx;t)$ 
(\mbox{$\bfx\in\CC^n$}, $t\in\CC$), 
that interpolates between two
systems---the \demph{start system} when $t=1$ and the \demph{target system} when $t=0$---in a particular way.
We require that $\calV(H)\subset\CC^n_{\bfx}\times\CC_t$ contains a curve $C$ that is a union of components of
$\calV(H)$ which projects dominantly to~$\CC_t$, and that $t=1$ is a regular value of this projection
$\pi\colon C\to\CC_t$.
We further require that~$C$ is bounded above a neighborhood of $1\in\CC_t$, and that  $\calV(H)$ is smooth at
$W\subset\pi^{-1}(1)$. 
The start system $H(\bfx,1)=0$ has $W$ among its isolated solutions.
The target system is $H(\bfx,0)=0$ and its intended solutions are the points of $C$ above $t=0$.

Given a homotopy $H$, we restrict $C$ to its points above the interval $[0,1]$ or above an arc in $\CC_t$ with
endpoints $\{0,1\}$. 
This gives a set of $|W|$ arcs in $\CC^n_{\bfx}\times\CC_t$, one for each point of $W$.
Each arc is either unbounded for $t$ near $0$ or it ends in a point of $\pi^{-1}(0)$.
Starting with points of $W$ and using numerical path-tracking to follow the corresponding arcs  will recover the isolated points of
$\pi^{-1}(0)$. 
In this way, we use the solutions $W$ of the start system to compute the solutions of the target system.
For more, see~\cite{Morgan,SW05}.

\subsection{Witness sets and numerical irreducible decomposition}\label{SS:witnessSets}
Let  $Y$ be an  irreducible subvariety of projective space $\PP^n$.
By Bertini's Theorem~\cite{jouanolou}, the dimension \defcolor{$\dim(Y)$} of $Y$ is the maximum number of general linear
polynomials that have a common zero on~$Y$, and its degree \defcolor{$\deg(Y)$} is the number of such common zeroes. 
For  a collection $L$ of $\dim(Y)$ general linear polynomials, the set $Y\cap\calV(L)$ of $\deg(Y)$ common zeroes is 
a \demph{linear section} of $Y$, 
called a \demph{witness point set} of $Y$.
If $F$ is a finite set of polynomials with $Y$ an irreducible component of $\calV(F)$, then the triple
$(F,L,Y\cap\calV(L))$ is a \demph{witness set} for $Y$. 

Suppose that $X\subset\PP^n$ is a union of irreducible components of $\calV(F)$.
A witness set for $X$ is composed of witness sets for each irreducible component of $X$.
We assume for simplicity that the linear sections are chosen coherently: Let $\defcolor{\ell_1},\dotsc,\defcolor{\ell_n}$
be general linear polynomials on $\PP^n$, and for each $e\in\{0,1,\dotsc,n\}$, set $\defcolor{L^e}:=(\ell_1,\dotsc,\ell_e)$.
For each dimension $e$, the $e$th witness set for $X$ is the triple  $(F,L^e,P^e)$ where $P^e$ is the set of
isolated points in $X\cap\calV(L^e)$.
If $X$ is \demph{equidimensional} of dimension $e$ (all components of $X$ have dimension $e$) then  $(F,L^e,X\cap\calV(L^e))$
is a witness set for $X$.
For this assertion/definition the generality of the $\ell_i$ is essential, by Bertini's Theorem.

\begin{rem}\label{r:move-witness}
 Given another collection $L'$ of $e$ linear polynomials, the convex combination $tL^e+(1-t)L'$ may be used in a homotopy
 $H(t)=(F,tL^e+(1-t)L')$ to transform the witness point set $P^e\subset X\cap\calV(L^e)$ into one lying in
 $X\cap\calV(L')$. 
 This homotopy can be used, for example, to test membership.
 In particular, if $X$ is equidimensional of dimension $e$, 
 $\bfx\in\PP^n$, and $L'$ is $e$ general linear polynomials 
 vanishing at $\bfx$, then $\bfx\in X$ if and only if~$\bfx$ is an endpoint of the homotopy $H(t)$
 with start points $X\cap\calV(L^e)$.\hfill$\diamond$
\end{rem}

A fundamental algorithm involving witness sets is numerical irreducible decomposition.
It first decomposes a witness set for $X$ into witness sets for $X_0,\dotsc,X_n$, where $\defcolor{X_e}\subset X$ is the
union of the irreducible components of $X$ of dimension $e$. 
When $X=X_e$, numerical irreducible decomposition computes the partition of $X\cap\calV(L)$ $(L=L^e)$ into subsets, each of
which is a linear section $Y\cap\calV(L)$ of an irreducible component $Y$ of $X$.

Numerically following the points of $X\cap\calV(L)$ as $L$ varies in a loop gives a monodromy permutation
\defcolor{$\omega$} of $X\cap\calV(L)$.  
The points belonging to a cycle of $\omega$ lie in the same irreducible component of $X$, and thus the cycles of
$\omega$ give a finer partition than the numerical irreducible decomposition.
Computing additional monodromy permutations coarsens this partition.
This monodromy break up algorithm~\cite{SVW_monodromy} gives a partition
$P_1\sqcup \dotsb\sqcup P_s$ of $X\cap\calV(L)$, where each $P_i\subset Y\cap\calV(L)$
for some irreducible component $Y$ of $X$.

The trace test~\cite{traceLRS,SVW_trace} is a heuristic stopping criterion for monodromy break up.
In it, the points of some part $P_i$ of the partition are numerically continued as $L$ moves in a general linear pencil. 
The average of the points in $P_i$ is collinear if and only if $P_i$ is a witness point set of a component.
Thus, when each part of the partition passes this \demph{trace test}, we have computed the numerical irreducible
decomposition.

This is unchanged if we replace projective varieties by affine varieties.
In practice, the algorithm operates on affine varieties, working in a random affine patch of $\PP^n$.

\subsection{Multiprojective varieties}\label{SS:Multiproj}
For more background, see~\cite[Ch.~8]{MS05}.
Let $k,n_1,\dotsc,n_k$ be positive integers and let $\defcolor{\PP^{\ndot}}:=\PP^{n_1}\times\dotsb\times\PP^{n_k}$ be
the indicated product of projective spaces.
Writing \defcolor{$\bfx_{i}$} for the indeterminates $x_{i,0},\dotsc,x_{i,n_i}$, we have that $\CC[\bfx_i]$ is the
homogeneous coordinate ring of $\PP^{n_i}$ and $\defcolor{\CC[\bfx]}:=\CC[\bfx_1,\dotsc,\bfx_k]$ is the coordinate ring of
$\PP^{\ndot}$.
This ring is \demph{multigraded}, its multihomogeneous elements $f(\bfx)$ are separately homogeneous in each variable
group $\bfx_1,\dotsc,\bfx_k$.
Such an element has a \demph{multidegree} which is a vector $(d_1,\dotsc,d_k)\in\NN^k$ where $d_i$ is the degree of
$f(\bfx)$ in the variable group $\bfx_i$.

A subvariety $X\subset\PP^{\ndot}$ (a \demph{multiprojective variety}) is a union of irreducible components of a set
$\calV(F)$, where $F\subset\CC[\bfx]$ is a finite set of multihomogeneous polynomials.
Each irreducible component $Y$ of $X$ has an intrinsic dimension $\dim(Y)$ as an algebraic variety.
As a subvariety of $\PP^{\ndot}$, its (extrinsic) dimension and degree are more
involved than for projective varieties.
This already occurs for hypersurfaces.
A multihomogeneous linear polynomial in $\CC[\bfx]$ has multidegree $(0,\dotsc,1,\dotsc,0)$:
it is linear in one variable group $\bfx_i$ and no other variables occur in it.
In particular, there are $k$ different types of `hyperplanes'.

There are similarly many different types of `linear' sections of multiprojective
varieties in $\PP^{\ndot}$.
Set $\defcolor{[\ndot]}:=\{(e_1,\dotsc,e_k)\in\NN^k\mid e_i\in\{0,1,\dots,n_i\}\}$ and let $\bfe\in[\ndot]$.
For each $i=1,\dotsc,k$, let $L_i$ be $e_i$ general linear polynomials in $\CC[\bfx_i]$ and write
$\bL^\bfe=(L_1,\dotsc,L_k)$.
Then $\calV(\bL^\bfe)\subset\PP^{\ndot}$ is a product of linear subspaces in the factors of
$\PP^{\ndot}$, where   the linear subspace in $\PP^{n_i}$ has dimension $n_i{-}e_i$.
When $Y\subset\PP^{\ndot}$ is an irreducible multiprojective variety with intrinsic dimension $\dim(Y)$, Bertini's
Theorem implies that $Y\cap\calV(\bL^{\bfe})$ is nonempty and finite only if $\dim(Y)=e_1+\dotsb+e_k=:\defcolor{|\bfe|}$.
Similarly, it is empty if $\dim(Y)<|\bfe|$ and, for $\dim(Y)>|\bfe|$, it is either empty or infinite.

The \demph{(multi)dimension} \defcolor{$\Dim(Y)$} of an irreducible multiprojective variety $Y\subset\PP^{\ndot}$ is the
set of vectors $\bfe\in[\ndot]$ such that $Y\cap\calV(\bL^\bfe)$ is finite and nonempty.
In~\cite{CLZ} this is called the support of $Y$.
Unlike for projective varieties, $\Dim(Y)$ is a set.
Note that $\bfe\in\Dim(Y)$ implies that $|\bfe|=\dim(Y)$.
The \demph{(multi)degree} of $Y$ is the map $\defcolor{\Deg_Y}\colon\Dim(Y)\to\NN$, where
$\Deg_Y(\bfe)$ is the number of points in the linear section $Y\cap\calV(\bL^\bfe)$.
For convenience, we extend the domain $\Deg_Y$ to $[\ndot]$, where if $\bfe\not\in\Dim(Y)$, then $\Deg_Y(\bfe)=0$.

If $X\subset\PP^{\ndot}$ has irreducible decomposition $X=Y_1\cup\dotsb\cup Y_s$, then we
define $\Deg_X$ by
\[
  \defcolor{\Deg_X}(\bfe)\ :=\ \sum_{j=1}^s \Deg_{Y_j}(\bfe)
  \qquad \mbox{ for }\bfe\in[\ndot]\,.
\]
%
%
Likewise, the dimension of $X$ is the support of $\Deg_X$,
\[
  \defcolor{\Dim(X)}\ =\ \{\bfe\in[\ndot]\mid \Deg_X(\bfe)>0\}
  \ =\ \bigcup_{j=1}^s \Dim(Y_j)\,.
\]
When $k=1$, this reduces to the dimension and degree of a projective variety $X\subset\PP^n$, where the
dimension of $X$ is the set of dimensions of its irreducible components and the degree sends $e$ to the degree
of the equidimensional part of $X$ of dimension $e$.

\begin{rem}\label{R:intTheory}
  The structure of the extrinsic dimension and degree of a multiprojective variety is a
 consequence of the structure of the homology groups of $\PP^{\ndot}$~\cite{Fulton}.
 The homology of $\PP^n$ has a $\ZZ$-basis \defcolor{$T^e$} for $e=0,\dotsc,n$, where $T^e$ is the class $[\Lambda^e]$ of a
 linear subspace $\Lambda^e$ of dimension $e$.
 Then the class of a subvariety $X\subset\PP^n$ is
 \[
   [X]\ =\ \sum_{e=0}^n \Deg_X(e) T^e\ .
 \]
 The homology of $\PP^{\ndot}$ has a $\ZZ$-basis $\defcolor{\bfT^\bfe}:=[\Lambda_1^{e_1}\times\dotsb\times\Lambda_k^{e_k}]$
 for $\bfe\in[\ndot]$, where $\Lambda_i^{e_i}\subset\PP^{n_i}$ is a linear space of dimension $e_i$.
 Then the class of a multiprojective variety $X\subset\PP^{\ndot}$ is
 \[
   [X]\ =\ \sum_{\bfe\in[\ndot]} \Deg_X(\bfe) \bfT^\bfe\ .\eqno{\diamond}
 \]
\end{rem}

A \demph{witness collection} for an irreducible multiprojective variety $Y\subset\PP^{\ndot}$ that is a component of
$\calV(F)$ is a map that assigns each $\bfe\in\Dim(Y)$ to $(F,\bL^{\bfe},Y\cap\calV(\bL^{\bfe}))$.
This triple is an \demph{$\bfe$-witness set of $Y$}
with $Y\cap\calV(\bL^{\bfe})$ an \demph{$\bfe$-witness point set of $Y$}.
As with ordinary witness sets, we assume that the linear polynomials are chosen coherently.
That is, for each $i\in\{1,\dots,k\}$, let $\defcolor{\ell_{i,1}},\dotsc,\defcolor{\ell_{i,n_i}}\in\CC[\bfx_i]$ be general linear
polynomials. 
For $\bfe\in[\ndot]$, set $\defcolor{L_i^{e_i}}:=(\ell_{i,1},\dotsc,\ell_{i,e_i})$ and
$\bL^\bfe:=(L_1^{e_1},\dotsc,L_k^{e_k})$.
If $X$ is a union of components of $\calV(F)$, then a witness collection for $X$ is the map that sends $\bfe\in\Dim(X)$ to
$(F,\bL^\bfe,P^\bfe)$, where~$P^\bfe$ is the set of isolated points of $X\cap\calV(\bL^{\bfe})$.

\begin{rem}\label{r:move-witness-muttiprojective}
 Given another collection $\bL'$ of $\bfe$ linear polynomials with $e_i$ in $\CC[\bfx_i]$, the convex combination
 $t\bL^\bfe+(1-t)\bL'$ may be used in a homotopy $H(t)=(F,t\bL^\bfe+(1-t)\bL')$ to transform the witness point set
 $P^\bfe\subset X\cap\calV(\bL^\bfe)$ into one lying in $X\cap\calV(\bL')$. 
 Similar to Remark~\ref{r:move-witness}, this homotopy can be used, for example, to test membership.\hfill$\diamond$
\end{rem}

 This membership test for multiprojective varieties relies on the result that if $X$ is irreducible and
 $\bfx\in\PP^{\ndot}$, then $\bfx\in X$ if and only if  there exists $\bfe\in\Dim(X)$ such that $\bfx$ is an endpoint of
 the homotopy $H(t)=(F,t\bL^\bfe+(1-t)\bL')$ with start points $X\cap\calV(\bL^\bfe)$, where  
 $\bL'$ is $\bfe$ general linear polynomials vanishing at $\bfx$.

\begin{alg}[Membership test for multiprojective varieties~\protect{\cite[Alg.~3]{HR15}}]
  \label{A:membership-multi}
  \mbox{\ }\newline
  {\bf Input:} Witness collection for an irreducible
  multiprojective variety $X\subset\PP^{\ndot}$
  and \mbox{$\bfx\in\PP^{\ndot}$}. \newline
  {\bf Output:} A boolean $B_\bfx$ which answers if $\bfx\in X$. \newline
  {\bf Do:} For each $\bfe\in\Dim(X)$, choose $\bL'$ to be $\bfe$ general linear polynomials vansihing at $\bfx$ and
  return ``true'' if  $\bfx$ is an endpoint of the homotopy $H(t) = (F,t \bL^\bfe + (1-t) \bL')$ with start points
  $X\cap\calV(\bL^{\bfe})$.
  Return ``false'' after testing all possible $\bfe\in\Dim(X)$.
\end{alg}

%

\subsection{Multiaffine varieties}\label{SS:MultiAffine}
Let $X\subset\PP^{\ndot}$ be a multiprojective variety.
Choosing an affine patch $\CC^{n_i}\subset\PP^{n_i}$ in each factor, 
$\defcolor{X_{\aff}}:= X\cap(\CC^{n_1}\times\dotsb\times\CC^{n_k})$  is an affine variety that retains much 
information about $X$.
To keep track of its multiprojective origins, we retain the decomposition $\CC^{n_1}\times\dotsb\times\CC^{n_k}$ 
from the factors of $\PP^{\ndot}$.
Write $\defcolor{\CC^{\ndot}}$ for $\CC^{n_1}\times\dotsb\times\CC^{n_k}$ and call a subvariety of $\CC^{\ndot}$ a
\demph{multiaffine variety}.
Algorithms for a multiprojective variety~$X$ operate locally on a corresponding multiaffine variety
$X_{\aff}$.

Let $i\in\{1,\dotsc,k\}$.
The affine patch $\CC^{n_i}$ has coordinate ring the polynomial ring $\CC[\bfy_i]$ with variables
$\defcolor{\bfy_i}:=(y_{i,1},\dotsc,y_{i,n_i})$.
This ring is not graded.
The coordinate ring of $\CC^{\ndot}$ is $\CC[\bfy]:=\CC[\bfy_1,\dotsc,\bfy_k]$.
This is an ordinary polynomial ring whose only structure is the indicated grouping of its variables.
A multihomogeneous polynomial $f(\bfx)\in\CC[\bfx]$ (multi)dehomogenizes to a polynomial $f(\bfy)\in\CC[\bfy]$.

The dimension of a multiaffine variety $X\subset\CC^{\ndot}$ is a set $\Dim(X)\subset[\ndot]$.
Its degree is a map $\Deg_X\colon [\ndot]\to\NN$.
These are defined in the same way as for multiprojective varieties, except that a homogeneous linear polynomial
$\ell(\bfx_i)\in\CC[\bfx_i]$ is replaced by its dehomogenization $\ell(\bfy_i)\in\CC[\bfy_i]$, which is a degree one
polynomial, or \demph{affine form}.
When the multiaffine patch $\CC^{\ndot}\subset\PP^{\ndot}$ is general, $\Dim(X)=\Dim(X_{\aff})$ and
$\Deg_X=\Deg_{X_{\aff}}$.

There is a second and more important reason (besides that our algorithms operate on them) to introduce multiaffine
varieties.
A key step in our numerical irreducible decomposition for multiprojective varieties in $\PP^{\ndot}$, called coarsening and described in Section~\ref{sec:coarse}, 
requires passing to a multiaffine variety (multi-dehomogenizing) and then rehomogenizing it into a different
multiprojective variety in a different multiprojective space.

\subsection{Monodromy and partial witness collections}\label{SS:partial}
In~\cite{HR15}, algorithms based on regeneration~\cite{Regen} were given to compute a
witness collection of a multiprojective variety.
We describe an alternative method based on monodromy.
Let $Y\subset\PP^{\ndot}$ be an irreducible component of $\calV(F)$, where $F\subset\CC[\bfx]$ is a finite set of
multihomogeneous polynomials.
Suppose that $\ell_{ij}\in\CC[\bfx_i]$ are general linear polynomials as in Subsection~\ref{SS:Multiproj}.
A \demph{partial witness collection} for $Y$ is a map $\Dim(Y)\ni\bfe\mapsto(F,\bL^\bfe,W_{\bfe})$, where
$W_\bfe\subset Y\cap\calV(\bL^\bfe)$ and at least one set $W_\bfe$ is nonempty.

The monodromy solving algorithm~\cite{Monodromy} gives a method to complete a partial witness set
to a witness set.
If in Subsection~\ref{SS:witnessSets}, we have a variety $X\subset\PP^n$ of pure dimension $e$ and a partial witness set
$W\subset X\cap\calV(L)$ ($L$ consists of $e$ linear polynomials, with the intersection transverse),
following points of $W$ as $L$ varies along loops both finds more points of $X\cap\calV(L)$ and computes a putative
numerical irreducible decomposition, with the caveat that the points found and subsequent decomposition will only lie on
the irreducible components of $X$ that contained points in the original set $W$.
The transversality of $X\cap\calV(L)$ at points of $W$ is necessary for there to be a homotopy starting at points of $W$.

This also may begin with a nonempty partial $\bfe$-witness set $W_{\bfe}\subset X\cap\calV(\bL^\bfe)$
of a multiprojective or multiaffine variety $X$ with $\bfe\in\Dim(X)$.
That is, monodromy may be used to complete $W_\bfe$ to a full $\bfe$-witness set $X\cap\calV(\bL^\bfe)$, at least for the
components of $X$ that contain points of $W_\bfe$.
In Section~\ref{sec:ldt}, we explain a more general procedure.

\subsection{Examples}
We give the dimension and multidegree of some multiaffine varieties.
Subsequent sections use these examples to demonstrate the numerical toolkit.

\begin{example}\label{Ex:AlgMatroid}
Let $Y\subset(\PP^1)^k$ be irreducible of intrinsic
dimension $e$.
Then $\Dim(Y)$ consists of $01$-vectors with $e$ 1s and $k{-}e$ 0s; the positions of the 1s give a subset of
$\{1,\dotsc,k\}$ of cardinality $e$.
For each such subset $\bfe$, let $\pi_\bfe\colon(\PP^1)^k\twoheadrightarrow(\PP^1)^e$ be the surjection onto the factors
corresponding to $\bfe$.
Our definitions imply that for $|\bfe|=e$, $\bfe\in\Dim(Y)$  if and only if
$\pi_\bfe \colon Y \to (\PP^1)^e$ is surjective.
Thus $\Dim(Y)$ is the algebraic matroid \cite[p.\ 211]{Oxley} of $Y_{\aff}\subset\CC^k$.
Its bases are subsets $y_{i_1},\dotsc,y_{i_e}$ of cardinality $e$ of the variables $y_1,\dotsc,y_k$ that are algebraically
independent in the coordinate ring of $Y_{\aff}$.\hfill$\diamond$
\end{example}

\begin{example}\label{Ex:Octahedron}
 We consider two multiaffine varieties in $\CC\times\CC\times\CC\times\CC$.
 Suppose that its coordinates are $x,y,z,w$ and consider the three polynomials
 \begin{eqnarray*}
  f&=& 1 + 2x + 3y^2 + 4z^3 + 5w^4\,,\\
  g&=& 1 + 2x + 3y + 5z + 7w\,, \ \mbox{ and}\\
  h&=&1+2x+3y+5z+7w+11xy+13xz+17xw+19yz+23yw+29zw\\
   &&+31xyz+37xyw+41xzw+43yzw+47xyzw\,.
 \end{eqnarray*}
 Let $X:=\calV(f,g)$ and $Y:=\calV(f,h)$, which are surfaces.
 Both have the same dimension, $\{1100, 1010, 1001, 0110, 0101, 0011\}$ (we omit commas).
 These form the second hypersimplex, which is an octahedron in their affine span.
 We display this in Figure~\ref{F:octahedron}.
 \begin{figure}[htb]
 \centering
   \begin{picture}(133,103)
     \put(20, 9){\includegraphics{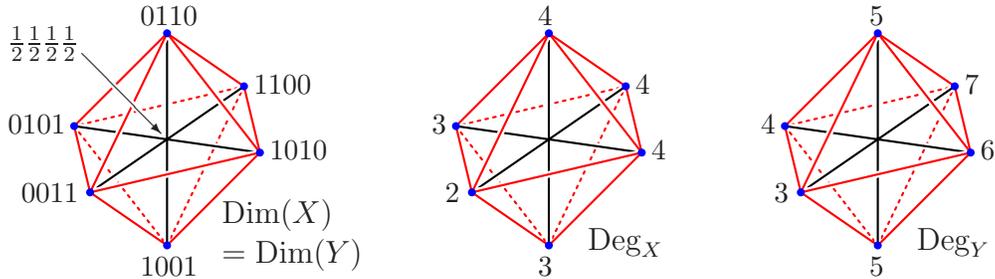}}
     \put(0,85){\small$\frac{1}{2}\frac{1}{2}\frac{1}{2}\frac{1}{2}$}
     \thicklines {\color{white} \put(27,84){\vector(1,-1){28}}
        \put(28,84){\vector(1,-1){28}}\put(29,84){\vector(1,-1){28}}}
     \thinlines \put(28,84){\vector(1,-1){30}}
     \put(50, 0){\small$1001$}     \put( 50,95){\small$0110$}
     \put( 5,28){\small$0011$}     \put( 93,69){\small$1100$}
     \put( 0,54){\small$0101$}     \put( 99,44){\small$1010$}
     \put(81,20){$\Dim(X)$}
     \put(81, 5){$=\Dim(Y)$}
   \end{picture}
   \qquad
   \begin{picture}(97,103)(16,0)
     \put(20, 9){\includegraphics{figures/Octahedron}}
     \put(56, 0){\small$3$}     \put( 56,95){\small$4$}
     \put(21,28){\small$2$}     \put( 93,69){\small$4$}
     \put(16,54){\small$3$}     \put( 99,44){\small$4$}
     \put(75,10){$\Deg_X$}
   \end{picture}
   \qquad
   \begin{picture}(97,103)(16,0)
     \put(20, 9){\includegraphics{figures/Octahedron}}
     \put(57, 0){\small$5$}     \put( 57,95){\small$5$}
     \put(21,28){\small$3$}     \put( 93,69){\small$7$}
     \put(16,54){\small$4$}     \put( 99,44){\small$6$}
     \put(75,10){$\Deg_Y$}
   \end{picture}
   \caption{Dimension and degree of multiaffine varieties.}
   \label{F:octahedron}
\end{figure}

\noindent Both $g$ and $h$ are the dehomogenization of multilinear polynomials on $(\PP^1)^4$.
 The difference between $\Deg_X$ and $\Deg_Y$ is that  $\calV(f,g)$ is a reducible variety in $(\PP^1)^4$ which has
 components not meeting the given multiaffine patch so that~$X$ is a component of $\calV(f,g)$ in $(\PP^1)^4$.
 In contrast, $h$ is sufficiently general so that $Y$ is dense in  $\calV(f,h)$ in
 $(\PP^1)^4$.~\hfill$\diamond$ 
\end{example}

\begin{example}\label{Ex:Richardson}
  Suppose that $\ndot=(3,3,3)$.
  Let $M=(y_{i,j})_{i,j=1}^3$ be a $3\times 3$ matrix with rows the variable groups $\bfy_1,\bfy_2,\bfy_3$
  of $\CC^{\ndot}$.
  Set $C:=(I_3 \mid M)^T$, a $6\times 3$ matrix, and let $N_1,N_2$ be general complex $6\times 2$ matrices.
  The conditions $\rank(C \mid N_i)\leq 4$ for $i=1,2$ define an irreducible subvariety $Y$ of $\CC^9$ of dimension five.
  (Taking the column span of $C$ parameterizes a dense open subset of the Grassmannian $G(3,6)$, each
  condition $\rank(C \mid N_i)\leq 4$ gives a codimension two Schubert variety, and these are in general position by the
  choice  of the $N_i$.
  Thus $Y$ is an open subset of a Richardson variety.)

  Each condition  $\rank(C \mid N_i)\leq 4$ is given by cubic determinants (minors) of the six $5\times 5$ matrices
  obtained by removing a row of $(C \mid N_i)$.
  Let \defcolor{$f_{i,j}$} be the minor when row $j$ is removed.
  It has degree one in each variable group $\bfy_1,\bfy_2,\bfy_3$, and so $Y$ is a multiaffine subvariety of
  $\CC^{\ndot}$.
  Its dimension is the set $\{\bfe\in[\ndot]\mid |\bfe|=5\}$, which consists of the twelve integer points in the hexagon on 
  the left below.
  On the right is its multidegree, where $\Deg_Y(\bfe)$ is displayed adjacent to $\bfe$.
\begin{equation}\label{Eq:Richardson}
  \raisebox{-50pt}{\begin{picture}(128,105)(-19,0)
    \put(-1.5,8.5){\includegraphics{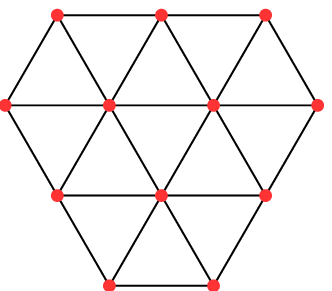}}
    \put(  0,92){\small$032$} \put(37,92){\small$131$} \put( 75,92){\small$230$}
    \put(-19,59){\small$023$}    \put( 93,59){\small$320$}
    \put( -4,32){\small$113$}    \put( 78,32){\small$311$}
    \put( 21, 0){\small$203$}    \put( 52, 0){\small$302$}
    \put(20,65){{\color{white}\rule{20pt}{10pt}}}  \put(21,66.5){\small$122$}
    \put(51,65){{\color{white}\rule{20pt}{10pt}}}  \put(52,66.5){\small$221$}
    \put(35,39){{\color{white}\rule{20pt}{10pt}}}  \put(37,40.5){\small$212$}
  \end{picture}
  \qquad
  \begin{picture}(98,100)(-2,-1)
    \put(-.5,7.5){\includegraphics{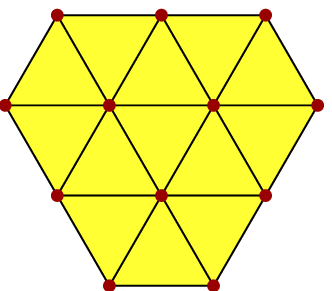}}

   \put(13,92){{\color{blue}$1$}} \put(43,92){{\color{blue}$2$}}\put(73,92){{\color{blue}$1$}}
   \put(-2,68){{\color{blue}$1$}} \put(28,68){{\color{blue}$3$}}
            \put(58,68){{\color{blue}$3$}} \put(88,68){{\color{blue}$1$}}
    \put(13,42){{\color{blue}$2$}} \put(43,42){{\color{blue}$3$}}\put(73,42){{\color{blue}$2$}}
    \put(28,16){{\color{blue}$1$}} \put(58,16){{\color{blue}$1$}}
  \end{picture}}
\end{equation}
 Replacing the twelve minors $f_{i,j}$ defining the rank conditions by the subset $f_{1,3},f_{1,5},f_{2,4},f_{2,6}$
 gives a complete intersection with four components, one of which is $Y$. 
 Two have the same dimension as $Y$ and one has a different dimension.
 We display their multidegrees below.
\[
   \begin{picture}(98,91)(-2,6.5)
    \put(-.5,6.5){\includegraphics{figures/Richardson}}

   \put(13,91){{\color{blue}$1$}} \put(43,91){{\color{blue}$1$}}\put(73,91){{\color{blue}$1$}}
   \put(-2,67){{\color{blue}$1$}} \put(28,67){{\color{blue}$3$}}
            \put(58,67){{\color{blue}$3$}} \put(88,67){{\color{blue}$1$}}
    \put(13,41){{\color{blue}$1$}} \put(43,41){{\color{blue}$3$}}\put(73,41){{\color{blue}$1$}}
    \put(28,15){{\color{blue}$1$}} \put(58,15){{\color{blue}$1$}}
  \end{picture}
  \qquad
   \begin{picture}(98,91)(-2,6.5)
    \put(-.5,6.5){\includegraphics{figures/Richardson}}

   \put(13,91){{\color{blue}$1$}} \put(43,91){{\color{blue}$1$}}\put(73,91){{\color{blue}$1$}}
   \put(-2,67){{\color{blue}$1$}} \put(28,67){{\color{blue}$3$}}
            \put(58,67){{\color{blue}$3$}} \put(88,67){{\color{blue}$1$}}
    \put(13,41){{\color{blue}$1$}} \put(43,41){{\color{blue}$3$}}\put(73,41){{\color{blue}$1$}}
    \put(28,15){{\color{blue}$1$}} \put(58,15){{\color{blue}$1$}}
  \end{picture}
  \qquad
   \begin{picture}(98,91)(-2,6.5)
    \put(-.5,6.5){\includegraphics{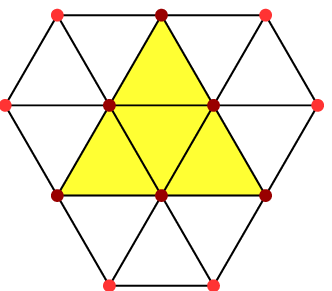}}

  \put(43,91){{\color{blue}$1$}}
   \put(28,67){{\color{blue}$2$}}   \put(58,67){{\color{blue}$2$}} 
    \put(13,41){{\color{blue}$1$}} \put(43,41){{\color{blue}$2$}}\put(73,41){{\color{blue}$1$}}
  \end{picture}
\]
   
\end{example}

\subsection{Numerical irreducible decomposition for multiprojective varieties}\label{S:I.7}
An algorithm for computing witness set collections was given in~\cite{HR15}.
There, Example~20 showed that the trace test cannot be applied to a witness set collection for $X\subset\PP^{\ndot}$.
We must embed $X$ into an affine or projective space and transform the witness set collection into a witness set for 
the embedded $X$, and then apply the trace test.

This poses several problems.
Under the Segre embedding, $X\subset\PP^{\ndot}$ becomes a subvariety $\sigma(X)$ of $\PP^N$, where
$N+1=(n_1+1)\dotsb(n_k+1)$.
Following~\cite[Exer.~19.2]{Harris}, if $X$ has dimension $d$, then $\sigma(X)$ has degree
 \begin{equation}\label{Eq:Segre-degree}
  \sum_{|\bfe|=d} \binom{d}{\bfe} \Deg_X(\bfe)\,,
 \end{equation}
where $\binom{d}{\bfe}$ is the multinomial coefficient $\frac{d!}{e_1!\dotsb e_k!}$.
Thus, both the ambient dimension and size of a witness set increases dramatically.

Replacing $X$ by its intersection with an affine patch $X_{\aff}\subset\CC^{n_1+\dotsb+n_k}$, does not increase its ambient
dimension. 
Unlike the Segre embedding, it is not clear how to efficiently transform a witness collection for $X$ into a witness set
for $X_{\aff}$.
The Richardson variety $Y$ of Example~\ref{Ex:Richardson} has degree 450 under the Segre map and degree eight as an affine
variety.


\section{Computing dimension and completing a partial witness set}\label{sec:ldt}

Suppose that $X\subset\CC^n$ is an irreducible affine variety that is a component of $\calV(F)$, for a collection
$F=(f_1,\dotsc,f_m)$ of polynomials.
We assume that $\calV(F)$ is reduced along $X$ in that there is a point $x\in X$ such that the differential
$\defcolor{d_xF}:=(d_xf_1,\dotsc,d_xf_m)$ (a linear map $\CC^n\to\CC^m$) has rank $n-\dim(X)$.
Then $X$ is smooth at $x$ with tangent space \defcolor{$T_xX$} the kernel of $d_xF$.
The smooth points of $X$ form a nonempty Zariski open subset.

The differential $d_xF$ at a general smooth point $x\in X$ is
given by the Jacobian matrix of $F$,
\[
  \defcolor{DF}\ :=\ \left(\partial f_i/\partial x_j\right)_{i=1,\dotsc,m}^{j=1,\dotsc,n}\ ,
\]
evaluated at $x$.
Thus $\dim(X)=n-\rank(DF(x))$.

\subsection{Dimension of an irreducible multiprojective variety}
Let $X$ be an irreducible subvariety of $\PP^{\ndot}$ of intrinsic dimension $e$.
Its dimension $\Dim(X)$ is a subset of
 \begin{equation}\label{Eq:Dims}
   \{\bfe\in[\ndot] \mid |\bfe| = e\}\,.
 \end{equation}
Castillo et al.~\cite{CLZ} characterized $\Dim(X)$ as follows.
For $\defcolor{I}=\{i_1,\dotsc,i_s\}\subset\{1,\dotsc,k\}$, define 
$\defcolor{\PP^{n_I}}:=\PP^{n_{i_1}}\times\dotsb\times\PP^{n_{i_s}}$, and let
$\defcolor{\pi_I}\colon\PP^{\ndot}\twoheadrightarrow\PP^{n_I}$ be the projection onto the factors indexed by $I$.
Let \defcolor{$\dim_I(X)$} be the intrinsic dimension of $\pi_I(X)\subset\PP^{n_I}$.
Dimension counting implies that if $\bfe\in\Dim(X)$, then
 \begin{equation}\label{Eq:CLZ}
    e_{i_1}+\dotsb +e_{i_s}\ \leq\ \dim_I(X)\,.
 \end{equation}
This follows because if $\ell(\bfx_i)$ is a linear polynomial in the variable group $\bfx_i$, then
$\calV(\ell(\bfx_i))$ is $\pi_{\{i\}}^{-1}(\calV(\ell(\bfx_i)))$, with the second variety  $\calV(\ell(\bfx_i))$ a hyperplane in
$\PP^{n_i}$.

\begin{proposition}[Thm.~1.1 in~\cite{CLZ}]\label{P:CLZ}
  Suppose that $X\subset\PP^{\ndot}$ is an irreducible multiprojective variety.
  Then $\bfe\in[\ndot]$ lies in $\Dim(X)$ if and only if $|\bfe|=\dim(X)$ and for all proper subsets~$I$ of
  $\{1,\dotsc,k\}$, the inequality~\eqref{Eq:CLZ} holds.
\end{proposition}

These inequalities in $\RR^k$ define a lattice polytope of dimension at most $k{-}1$, which is a polymatroid polytope
(called a generalized permutahedron in~\cite{Postnikov}).

\begin{example}\label{Ex:furtherDecomposition}
 We continue Example~\ref{Ex:Richardson}.
 Suppose that in addition to the four minors defining the reducible complete intersection
 $X\subset \PP^3\times\PP^3\times\PP^3$, defining polynomials $F$ include the quadrics
\[  
   y_{1,1}y_{2,2}-y_{1,2}y_{2,1}\quad \mbox{ and }\quad
   y_{1,1}y_{2,3}-y_{1,3}y_{2,1}\,.
\]  
 Then $\calV(F)$ has intrinsic dimension three with twelve irreducible components---the four components of $X$ giving rise to  2, 3,
 3, and 4 irreducible components, respectively.
 The $i$th row of Figure~\ref{F:splittingApart} displays the dimension and multidegree of the irreducible decomposition of
 $Y\cap\calV(y_{1,1}y_{2,2}-y_{1,2}y_{2,1},y_{1,1}y_{2,3}-y_{1,3}y_{2,1})$, where $Y$ is the $i$th component of $X$ from
 Example~\ref{Ex:Richardson}.  
\begin{figure}[htb]
\centering
 \begin{picture}(85,74)(-9,0)
    \put(-1.5,9.5){\includegraphics{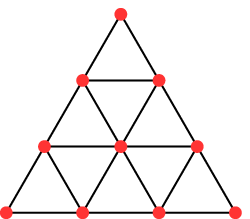}}
             \put(25,72){\small$030$} 
    \put( 3,47){\small$021$}\put( 47,47){\small$120$}
    \put(-9,27){\small$012$}                         \put( 59,27){\small$210$}
    \put(-9, 0){\small$003$}\put( 14, 0){\small$102$}\put( 36, 0){\small$201$}\put(58, 0){\small$300$}
   \put(23,33){{\color{white}\rule{20pt}{10pt}}}  \put(25,34.5){\small$111$}
  \end{picture}
  \qquad
  \begin{picture}(78,74)(6,-13.5)
    \put(7.6,-4.5){\includegraphics{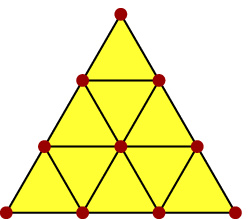}}
    \put(39.7,58){\small{\color{blue}$1$}} 
    \put(28,40.5){\small{\color{blue}$3$}}\put(52,40.5){\small{\color{blue}$2$}}
    \put(17,21.5){\small{\color{blue}$3$}}\put(40,22.5){\small{\color{blue}$5$}}\put( 63,21.5){\small{\color{blue}$2$}}
    \put( 6,2.5){\small{\color{blue}$1$}}\put(29,3.5){\small{\color{blue}$3$}}
    \put(51,3.5){\small{\color{blue}$3$}}\put(74,2.5){\small{\color{blue}$1$}}
  \end{picture}
  \qquad
  \begin{picture}(70,74)(6,-13.5)
    \put(7.6,-4.5){\includegraphics{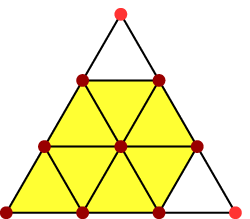}}
    \put(28,40.5){\small{\color{blue}$2$}}\put(52,40.5){\small{\color{blue}$1$}}
    \put(17,21.5){\small{\color{blue}$2$}}\put(40,22.5){\small{\color{blue}$3$}}\put( 63,21.5){\small{\color{blue}$1$}}
    \put( 6,2.5){\small{\color{blue}$1$}}\put(29,3.5){\small{\color{blue}$2$}}
    \put(51,3.5){\small{\color{blue}$2$}}
  \end{picture}\vspace{-1pt}

  \begin{picture}(78,74)(0,-10.5)
    \put(7.6,-4.5){\includegraphics{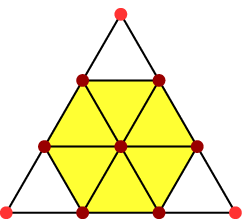}}
    \put(28,40.5){\small{\color{blue}$1$}}\put(52,40.5){\small{\color{blue}$1$}}
    \put(17,21.5){\small{\color{blue}$2$}}\put(40,22.5){\small{\color{blue}$3$}}\put( 63,21.5){\small{\color{blue}$1$}}
    \put(29,3.5){\small{\color{blue}$2$}}\put(51,3.5){\small{\color{blue}$1$}}
  \end{picture}
  \qquad
  \begin{picture}(78,74)(0,-10.5)
    \put(7.6,-4.5){\includegraphics{figures/DR3}}
    \put(28,40.5){\small{\color{blue}$1$}}\put(52,40.5){\small{\color{blue}$1$}}
    \put(17,21.5){\small{\color{blue}$1$}}\put(40,22.5){\small{\color{blue}$2$}}\put( 63,21.5){\small{\color{blue}$1$}}
    \put(29,3.5){\small{\color{blue}$1$}}\put(51,3.5){\small{\color{blue}$1$}}
  \end{picture}
  \qquad
  \begin{picture}(70,74)(0,-10.5)
    \put(7.6,-4.5){\includegraphics{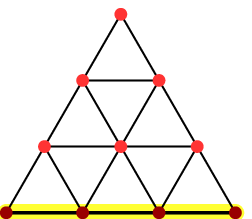}}
    \put( 6,2.5){\small{\color{blue}$1$}}\put(29,3.5){\small{\color{blue}$2$}}
    \put(51,3.5){\small{\color{blue}$2$}}\put(74,2.5){\small{\color{blue}$1$}}
  \end{picture}\vspace{-4pt}

  \begin{picture}(78,74)(0,-10.5)
    \put(7.6,-4.5){\includegraphics{figures/DR3}}
    \put(28,40.5){\small{\color{blue}$1$}}\put(52,40.5){\small{\color{blue}$1$}}
    \put(17,21.5){\small{\color{blue}$2$}}\put(40,22.5){\small{\color{blue}$3$}}\put( 63,21.5){\small{\color{blue}$1$}}
    \put(29,3.5){\small{\color{blue}$2$}}\put(51,3.5){\small{\color{blue}$1$}}
  \end{picture}
  \qquad
  \begin{picture}(78,74)(0,-10.5)
    \put(7.6,-4.5){\includegraphics{figures/DR3}}
    \put(28,40.5){\small{\color{blue}$1$}}\put(52,40.5){\small{\color{blue}$1$}}
    \put(17,21.5){\small{\color{blue}$1$}}\put(40,22.5){\small{\color{blue}$2$}}\put( 63,21.5){\small{\color{blue}$1$}}
    \put(29,3.5){\small{\color{blue}$1$}}\put(51,3.5){\small{\color{blue}$1$}}
  \end{picture}
  \qquad
  \begin{picture}(70,74)(0,-10.5)
    \put(7.6,-4.5){\includegraphics{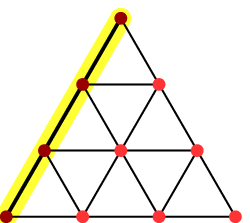}}
    \put(35.5,52.5){\small{\color{blue}$1$}} 
    \put(28,40.5){\small{\color{blue}$2$}}
    \put(17,21.5){\small{\color{blue}$2$}}
    \put( 6,2.5){\small{\color{blue}$1$}}
  \end{picture}\vspace{-4pt}

  \begin{picture}(78,74)(0,-10.5)
    \put(7.6,-4.5){\includegraphics{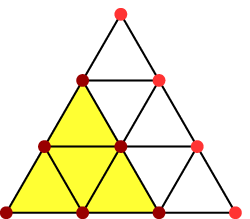}}
    \put(28,40.5){\small{\color{blue}$1$}}
    \put(17,21.5){\small{\color{blue}$2$}}\put(40,22.5){\small{\color{blue}$2$}}
    \put( 6,2.5){\small{\color{blue}$1$}}\put(29,3.5){\small{\color{blue}$2$}}
    \put(51,3.5){\small{\color{blue}$1$}}
  \end{picture}
  \quad
  \begin{picture}(78,74)(0,-10.5)
    \put(7.6,-4.5){\includegraphics{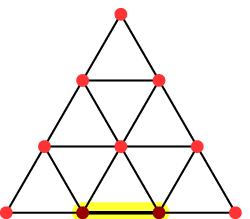}}
    \put(29,3.5){\small{\color{blue}$1$}} \put(51,3.5){\small{\color{blue}$1$}}
  \end{picture}
  \quad
  \begin{picture}(78,74)(0,-10.5)
    \put(7.6,-4.5){\includegraphics{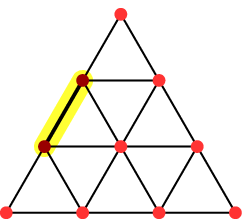}}
    \put(28,40.5){\small{\color{blue}$1$}}
    \put(17,21.5){\small{\color{blue}$1$}}
  \end{picture}
  \quad
  \begin{picture}(78,74)(0,-10.5)
    \put(7.6,-4.5){\includegraphics{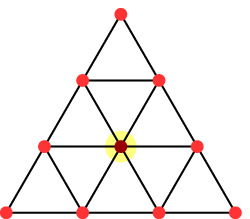}}
    \put(40,22.5){\small{\color{blue}$1$}}
  \end{picture}

  \caption{Decomposition of $X\cap\calV( y_{1,1}y_{2,2}-y_{1,2}y_{2,1},y_{1,1}y_{2,3}-y_{1,3}y_{2,1})$.}
  \label{F:splittingApart}
\end{figure}
 The first row also shows the set $\{\bfe\in[(3,3,3)]\mid |\bfe|=3\}$~from~\eqref{Eq:Dims}.

 As $k=3$, the dimension of a polymatroid polytope $\Dim(Z)$ is at most $2$.
 For seven components this is a polygon, for four, it is a line segment, and for one, it is a point.\hfill$\diamond$ 
\end{example}

Let $x$ be a point on an irreducible multiprojective variety $X\subset\PP^{\ndot}$ and 
suppose that 
\mbox{$I\subset\{1,\dots,k\}$}. 
We assume that $x$ is general in that the map $\pi_I$ is regular at~$x$.
(That is, $x$ is a smooth point of $X$ and the projection map $d_x\pi_I\colon T_xX\to T_{\pi_I(x)}\PP^{n_I}$ has
maximal rank among all smooth points of $X$.)
Then, $\dim_I(X)$ is equal to the dimension of~$d_x\pi_I(T_xX)$.

This leads to a method to compute these dimensions in local coordinates.
Suppose that $F=(f_1,\dotsc,f_m)$ are polynomials in $\CC[\bfy_1,\dotsc,\bfy_k]$
which are the dehomogenization of multihomogeneous polynomials defining $X\subset\PP^{\ndot}$ in some multiaffine patch
$\CC^\ndot\subset \PP^{\ndot}$ containing $x$.
Suppose that $Y$ is the component of $\calV(F)$ containing $x$ and $Y$ is smooth at $x$.
Then the intrinsic dimension $\dim(Y)$ of $Y$ (the local dimension of $\calV(F)$ at $x$) is the dimension of the tangent
space $T_xY$, which is the kernel of the Jacobian $DF(x)$ of $F$ at~$x$.
Thus, $\dim(Y)=\dim_x(Y)=\dim\ker DF(x).$ 

The variable groups $\bfy_1,\dotsc,\bfy_k$  partition the columns of the Jacobian matrix
\[
   DF\ =\ \left(\;\frac{\partial F}{\partial\bfy_1}\;\right|\;
          \left.\frac{\partial F}{\partial\bfy_2}\;\right|\ \dotsb\ 
          \left|\;\frac{\partial F}{\partial\bfy_k}\;\right)\ ,
\]
where for each $l\in\{1,\dotsc,k\}$,
$\defcolor{\partial F/\partial\bfy_l}= (\partial f_i/\partial y_{l,j})_{i=1,\dotsc,m}^{j=1,\dotsc,n_l}$ is the Jacobian matrix
with respect to the variables $\bfy_l$. 
%
%
Denote by $DF_{ I^c}$ 
 the submatrix of $DF$ obtained by omitting the blocks 
${\partial F}/{\partial\bfy_i}$ for $i\in I$. 
In other words,
 \begin{equation}\label{Eq:partialJacobian}
   \defcolor{DF_{  I^c }}\ 
   :=\ \left(\;\frac{\partial F}{\partial\bfy_{j_1}}\;\right|
        \ \dotsb\ 
         \left|\;\frac{\partial F}{\partial\bfy_{j_r}}\;\right),
         \quad \text{where }\quad
	 I^c =\{j_1,\dots,j_r\}.  
 \end{equation}
Since the intrinsic dimension of the image of $Y$ under $\pi_I$ equals the intrinsic dimension of $Y$ minus the  intrinsic
dimension of the fiber over a general point, it follows that 
$\dim_I(Y)=\dim \ker DF(x)-\dim \ker DF_{  I^c}(x)$.

By Proposition~\ref{P:CLZ}, if $Y$ is the irreducible component of a multiprojective variety
$\calV(F)\subset\PP^{\ndot}$ containing the point $x$, then $\Dim(Y)$ is determined by the numbers $\dim(Y)$ and $\dim_I(Y)$
for all proper subsets $I$ of $\{1,\dotsc,k\}$.
Let \defcolor{$\Dim_x(F)$} be these numbers, which may be computed in local coordinates by determining the ranks of the
Jacobian matrices $DF(x)$ and 
$DF_{ I^c }(x)$.
This leads to two algorithms to classify the dimension of components of $\calV(F)$ given points of $\calV(F)$.

\begin{alg}[Dimension at a smooth point]
  \label{A:local-dimension}
  \mbox{\ }\newline
  {\bf Input:} A general smooth point $x\in\calV(F)\subset\PP^{\ndot}$.\newline
  {\bf Output:} $\Dim_x(F)$.  \newline
  {\bf Do:} Dehomogenize $F$ and compute $\dim\ker DF(x)$.
   For each  proper subset $I$ of
  $\{1,\dotsc,k\}$ compute  $\dim\ker DF_{ I^c }(x)$ 
  to determine the difference 
   $\dim\ker DF(x) - \dim\ker DF_{  I^c }(x)$.
\end{alg}
\newcommand{\aDim}{\Delta} 

If $y$ is smooth but not general, then it can be perturbed via a homotopy to a general point.
(Recall that witness points are smooth.)
Given \mbox{$F\subset\CC[\bfy_1,\dotsc,\bfy_k]$} defining a multiaffine variety
$\calV(F)\subset\CC^{\ndot}$, this algorithm simply skips the dehomogenization.

A multiprojective variety $X$ is \demph{equidimensional} if all irreducible components have the same multidimension.
A multiprojective variety has a unique decomposition into equidimensional pieces.
Given a collection $W$ of general smooth points of $\calV(F)$, by computing the local dimension via Algorithm~\ref{A:local-dimension} one can sort the points by the equidimensional component of $\calV(F)$ on which they lie.
Let \defcolor{$\Dim(W)$} be the set of dimensions of components of $\calV(F)$ containing points of $W$.
For $\aDim\in\Dim(W)$, define $\defcolor{W_\aDim}:=\{w\in W\mid\Dim_w(F)=\aDim\}$.
These sets partition $W$ and form the \demph{equidimensional decomposition} of $W$,
\[
    W\ =\ \bigsqcup \{ W_\aDim \mid \aDim\in\Dim(W)\}\,.
\]

\begin{alg}[Equidimensional decomposition]\label{alg:equi}
  \mbox{\ }\newline
  {\bf Input:} A finite set $W\subset \calV(F)$ of general smooth points.\newline
  {\bf Output:} $\Dim(W)$ and the equidimensional partition of $W$. \newline
  {\bf Do:} For each $w\in W$, compute the local dimension $\Dim_w(F)$ of $\calV(F)$ at $w$ to get $\Dim(W)$ and for each
  $\aDim\in\Dim(W)$ let $W_\aDim=\{w\in W\mid\Dim_w(F)=\aDim\}$.
\end{alg}

It is important that the points of $W$ be general so that the maps $\pi_I$ are regular on $W$.

\subsection{Completing a partial witness collection}
A partial witness collection $(F,\bL^\bfe,W_\bfe)$ for a multiprojective variety $Y$ may be completed to a witness
collection using monodromy.
While this was sketched in Subsection~\ref{SS:partial}, it needs the definitions given in this section.

\begin{alg}[Completing a witness collection from a single point]\label{alg:completion}
  \mbox{\ }\newline
  {\bf Input:} A general smooth point $y$ on an irreducible multiprojective variety $Y$ that is a component of
  $\calV(F)$.\newline 
  {\bf Output:} A witness collection for $Y$. \newline
  {\bf Do:} Use Algorithm~\ref{A:local-dimension} to compute $\Dim(Y)$.
  Choose linear polynomials $\ell_{ij}\in\CC[\bfx_i]$ for $i\in\{1,\dots,k \}$ and $j=1,\dotsc,n_i$ that are general given that they vanish
  at $y$. 
  Using the $\ell_{ij}$ gives a partial witness collection
  $\{(F,\bL^\bfe,\{y\})\mid \bfe\in\Dim(Y)\}$ for $Y$.
  Use monodromy as in Subsection~\ref{SS:partial} to complete each partial $\bfe$-witness point set $\{y\}$ to the
  complete $\bfe$-witness point set $Y\cap\calV(\bL^\bfe)$.
\end{alg}

\begin{proof}[Proof of correctness]
  We note that this does not have a stopping criterion, and is therefore technically not an algorithm.
  Nevertheless, by the choice of $\ell_{ij}$, each intersection  $Y\cap\calV(\bL^\bfe)$ is transverse and contains $\{y\}$.
  Thus, letting the $\bL^\bfe$ vary in a loop gives a homotopy.
  The rest follows from the discussion in Subsection~\ref{SS:partial}.
\end{proof}

\section{Cartesian products}\label{sec:product}

Of the twelve irreducible components $Y$ of the variety $\calV(F)$ of Example~\ref{Ex:furtherDecomposition}, $\Dim(Y)$ was a
line segment for four and a point for one.
In these five cases, $\Dim(Y)$ was decomposable as a product of polymatroid polytopes.
We will show that if $Y\subset\PP^{\ndot}$ is an irreducible multiprojective variety for which $\Dim(Y)$ is such a product,
then $Y=Y'\times Y''$ is a Cartesian product of irreducible varieties in disjoint factors of $\PP^{\ndot}$, and the witness
sets for $Y$ are also products of witness sets for $Y'$ and $Y''$.

Fix $1\leq l< k$.  Let $\defcolor{\npdot}:=(n_1,\dotsc,n_l)$ and
$\defcolor{\nppdot}:=(n_{l+1},\dotsc,n_k)$  so that
\mbox{$\PP^{\ndot}=\PP^{\npdot}\times\PP^{\nppdot}$}.
If $Y'\subset\PP^{\npdot}$ and $Y''\subset\PP^{\nppdot}$ are irreducible varieties, then so is
\mbox{$Y'\times Y''\subset \PP^{\ndot}$}.
Its intrinsic dimension is the sum of the intrinsic dimensions of 
its factors, $\dim(Y'\times Y'')=\dim(Y')+\dim(Y'')$. 
Its multidimension has a similar decomposition, 
 \begin{eqnarray*}
   \Dim(Y'\times Y'')&=& \Dim(Y') \times \Dim(Y'')\\
   &=& \{ \bfe\in[\ndot] \mid \bfe=(\bfe',\bfe'') \mbox{ for }\bfe'\in\Dim(Y') \mbox{ and }\bfe''\in\Dim(Y'')\}\,,
 \end{eqnarray*}
as $[\ndot]= [\npdot]\times [\nppdot]$.
This is a consequence of the definition given in Subsection~\ref{SS:Multiproj} for the multidimension
of a multiprojective variety, applied to such a product.

For $(\bfe',\bfe'')\in\Dim(Y'\times Y'')$, suppose that $\bL^{\bfe'}\subset\CC[\bfx_1,\dotsc,\bfx_l]$ and
$\bL^{\bfe''}\subset\CC[\bfx_{l+1},\dotsc,\bfx_k]$ are general linear polynomials with corresponding
witness point sets $W_{\bfe'}=Y'\cap\calV(\bL^{\bfe'})$ for $Y'$ and  $W_{\bfe''}=Y''\cap\calV(\bL^{\bfe''})$ for $Y''$.
Then
\[
  W_{\bfe'}\times W_{\bfe''}\ =\ \left(Y'\times Y''\right)\cap \calV(\bL^{\bfe'},\bL^{\bfe''})
\]
is an $(\bfe',\bfe'')$-witness point set for the product $Y'\times Y''$.

More generally, let $I\subset\{1,\dotsc,k\}$ be a proper subset with complement $J$ so that
$\PP^\ndot= \PP^{n_I}\times\PP^{n_J}$.
Given irreducible multiprojective varieties $Y\subset\PP^{n_I}$ and $Z\subset\PP^{n_J}$, their product is a multiprojective variety
$Y\times Z\subset\PP^\ndot$.
We similarly have $\Dim(Y\times Z)=\Dim(Y)\times\Dim(Z)$, and witness 
point sets for $Y\times Z$ are products of
witness point sets for $Y$ and for $Z$.
This reduces to the previous discussion after reordering the factors of $\PP^\ndot$.

\begin{thm}\label{thm:Product}
 An irreducible multiprojective variety $X\subset\PP^\ndot$ is a Cartesian product $X=Y\times Z$ of multiprojective
 varieties $Y\subset\PP^{n_I}$ and $Z\subset\PP^{n_J}$ in disjoint factors of\/ $\PP^\ndot$ if and only if $\Dim(X)$ is
 the product of polymatroid polytopes $P\subset[n_I]$ and $Q\subset[n_J]$ with $\Dim(Y)=P$ and $\Dim(Z)=Q$.

 When this occurs, the multidegree $\Deg_X(\bfe',\bfe'')$ for $\bfe'\in\Dim(Y)$ and $\bfe''\in\Dim(Z)$ is the
 product $\Deg_Y(\bfe')\cdot\Deg_Z(\bfe'')$ of multidegrees and any $(\bfe',\bfe'')$-witness point set for $X$ is
 the product of corresponding witness point sets for $Y$ and for $Z$.
\end{thm}

\begin{proof}
 The forward direction of the first part is a consequence of the preceding discussion, as is the second part of the
 theorem (which follows from the cartesian product $X=Y\times Z$).
 For the reverse direction of the first part,  suppose that $\Dim(X)=P\times Q$, where $P\subset[n_I]$ and $Q\subset[n_J]$
 are polymatroid polytopes in disjoint factors of $[\ndot]$, so that $I\sqcup J=\{1,\dotsc,k\}$.
 Since $P$ and $Q$ are polymatroid polytopes, there are integers $p$ and $q$ such that if $\bfe'\in P$ and $\bfe''\in Q$,
 then $|\bfe'|=p$,  $|\bfe''|=q$, and $p+q=\dim(X)$.
 
 Let us study the map $\pi_I\colon X\to \PP^{n_I}$ whose image $\pi_I(X)$ has dimension $\dim_I(X)$.
 Let $\bfe'\in\Dim(\pi_I(X))\subset[n_I]$ and let $\bL^{\bfe'}\subset\CC[\bfx_i\mid i\in I]$ be $|\bfe'|=\dim_I(X)$ general
 linear polynomials so that $\pi_I(X)\cap\calV(\bL^{\bfe'})$ consists of $d=\Deg_{\pi_I(X)}(\bfe')$ points.
 Since in $\PP^{\ndot}$, 
we have 
 $\calV(\bL^{\bfe'})=\pi_I^{-1}(\calV(\bL^{\bfe'}))$, the intersection $X\cap\calV(\bL^{\bfe'})$ is
 nonempty and it consists of~$d$ fibers of the map $\pi_I\colon X\to \pi_I(X)$.
 By the generality of $\bL^{\bfe'}$, each fiber has dimension $\dim(X)-\dim_I(X)$.
 Then there is some $\bfe''\in[n_J]$ such that if $\bL^{\bfe''}\subset\CC[\bfx_j\mid j\in J]$ are
 $|\bfe''|=\dim(X)-\dim_I(X)$ general linear polynomials, then $X\cap\calV(\bL^{\bfe'})\cap\calV(\bL^{\bfe''})$~is~nonempty.

 This implies that $(\bfe',\bfe'')\in\Dim(X)$ and in particular that $\bfe'\in P$ and $\bfe''\in Q$ and that
 $\dim_I(X)=\dim(\pi_I(X))=p$.
 Similarly,  $\dim_J(X)=\dim(\pi_J(X))=q$.
 Since \mbox{$X\subset \pi_I(X)\times\pi_J(X)$} and both are irreducible of dimension $p+q$, they are equal.
\end{proof}

\begin{example}\label{Ex:product}
 Let us look at the last two components in the bottom row of Figure~\ref{F:splittingApart}.
 The third component $Y$ has $\Dim(Y)=\{0\}\times\{12,21\}$.
 Its ideal is generated by
 \begin{eqnarray*}
    &    \quad  y_{1,1}\,,\ y_{1,2}\,,\ y_{1,3}\,,\ 
      19y_{2,2}+46y_{2,3}\,,\ 19y_{3,2}+46y_{3,3}+34\,,
      &\\ &
    243y_{2,3}y_{3,1}-243y_{2,1}y_{3,3}-306y_{2,1}+1020y_{2,3}-342y_{3,1}+1194y_{3,3}+68\,&
 \end{eqnarray*}
 The first three define the point $\{(0,0,0)\}$ in the first $\CC^3$ factor and the next two define a plane in
 each of the last two factors.
 Thus $Y=\{(0,0,0)\}\times Z$, where $Z\subset\CC^2\times\CC^2$ is the hypersurface defined by the last bilinear
 polynomial. 
 This explains $\Dim(Y)$ and $\Deg_Y$.

 The last component $Y$ has $\Dim(Y)=\{1\}\times\{1\}\times\{1\}$.
 Its ideal is generated by 
 \begin{eqnarray*}
    &
     57y_{1,1}-199y_{1,3}\,,\  
     19y_{1,2}+46y_{1,3}\,,\ 
      57y_{2,1}-199y_{2,3}\,,
   &\\ &
     19y_{2,2}+46y_{2,3}\,,\  
     171y_{3,1}-597y_{3,3}-34\,,\ 
     19y_{3,2}+46y_{3,3}+34\,.&
 \end{eqnarray*}
 As there are two affine forms in each variable group, $Y$ is isomorphic to $\CC\times\CC\times\CC$, which again explains
 its multidegree.\hfill$\diamond$
\end{example}

Membership testing in Cartesian products can be simplified
since one can consider membership in each factor independently.  

\begin{alg}[Membership test in Cartesian product]
  \label{A:membership-product}
  \mbox{\ }\newline
  {\bf Input:} A witness collection for an irreducible
   multiprojective variety $X\subset\PP^{\ndot}$
   which is a Cartesian product $X = Y\times Z$ 
   of multiprojective varieties $Y\subset\PP^{n_I}$ and $Z\subset\PP^{n_J}$
   in disjoint factors of\/ $\PP^{\ndot}$ 
   and a point $\bfx = (\bfy,\bfz)\in\PP^{\ndot}$.
  \newline
  {\bf Output:} A triple $(B_{\bfx},B_{\bfy},B_{\bfz})$ 
of booleans such that $B_\omega$ answers if $\omega \in \Omega$.   
\newline
  {\bf Do:} Select $\bfe\in\Dim(X)$ and fix a point $(\bfy^*,\bfz^*)$
  from the $\bfe$-witness point set for $X$.  Construct witness collections
  for $Y$ and $Z$ from the given witness collection for $X$
  following Theorem~\ref{thm:Product} with polynomial systems
  $F(\bfy,\bfz^*)$ and $F(\bfy^*,\bfz)$, respectively.  
  Apply Algorithm~\ref{A:membership-multi} to $Y$ and $Z$
  yielding $B_{\bfy}$ and $B_{\bfz}$, respectively.  
  Set $B_\bfx = B_\bfy \times B_\bfz$
\end{alg}

\begin{proof}[Proof of correctness]
Since $X = Y\times Z$, we know 
$\bfx\in X$ if and only if $\bfy\in Y$ and $\bfz\in Z$.
Let $\bfe=(\bfe',\bfe'')\in\Dim(X)$ be the selection
that yielded $(\bfy^*,\bfz^*)$ in the $\bfe$-witness point set for $X$
with corresponding $\bL^\bfe = (\bL^{\bfe'},\bL^{\bfe''})$.  
Then, $Y\times\{z^*\}$ and $\{y^*\}\times Z$
are irreducible components of $\calV(F,\bL^{\bfe''})$
and $\calV(F,\bL^{\bfe'})$, respectively.  
Hence, by selecting a representative of $\bfy^*$ and $\bfz^*$, 
it follows that $Y$ and $Z$ are irreducible components of 
$F(\bfy,\bfz^*)$ and $F(\bfy^*,\bfz)$, respectively.
Hence, Algorithm~\ref{A:membership-product} decides
membership of $y$ in $Y$ and $z$ in $Z$ which immediately
decides membership of $x$ in $X$.
\end{proof}

A natural recursion applies when $X$ is a Cartesian product of  more than two varieties.

\section{Refining and coarsening witness collections}\label{sec:coarse}

Algorithms for computing witness sets and witness collections operate on affine patches of projective
and multiprojective varieties.
Changing the multiaffine structure is straightforward in such patches and corresponds to a birational map on
the underlying (multi)pro\-jec\-tive variety.
We describe this and investigate how it affects witness collections.

A multiaffine variety $X_{\aff}\subset\CC^{\ndot}$ is simply a variety in the affine space $\CC^{n_1+\dotsb+n_k}$
whose coordinates have been partitioned into subsets of sizes $n_1,\dotsc,n_k$.
Changing the partition does not change the variety $X_{\aff}$, but it does change its multiaffine structure,
that is, its multidimension and multidegrees.
In particular, repartitioning changes how $X_{\aff}$ is represented using a witness collection.
Any repartitioning is a composition of two operations, \demph{refining}, in which one variable group is
split into two, and \demph{coarsening}, in which two variable groups are merged into one.
We describe the geometry of refining and coarsening, and give algorithms for transforming witness collections for
both. 

\begin{example}\label{ex:cubicPerestroika}
The polynomial $y^2-2xy-x^3+x$ defines a plane cubic curve.
As a multiaffine variety in $\CC^1_x\times\CC^1_y$ its multidimension is $\{ 10,01\}$ with corresponding multidegrees~$2$ for $10$ and $3$ for $01$.
In $\CC^2$, it is represented by a witness set
which uses a linear section such as shown at center below.
In $\CC^1_x\times\CC^1_y$, it is represented by a witness collection, which are its intersections with a vertical and with a 
horizontal line as at right below.\hfill$\diamond$
\[
  \begin{picture}(130,103)(0,-12)
      \put(0,0){\includegraphics{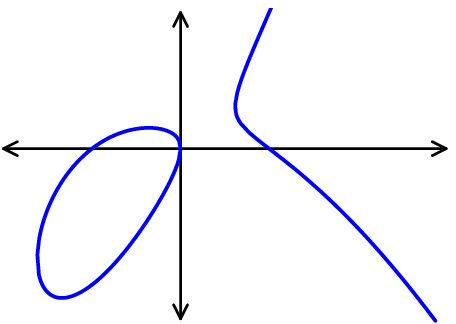}}
      \put(15,-12){$\calV(y^2{-}2xy{-}x^3{+}x)$}
   \end{picture}
      \quad
   \begin{picture}(130,103)(0,-12)
      \put(0,0){\includegraphics{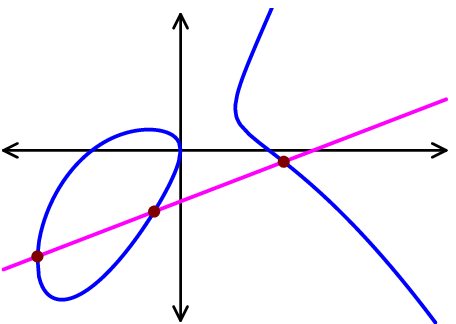}}
      \put(105,61){$L$}
      \put(35,-12){Linear Section}
    \end{picture}
  \quad
   \begin{picture}(130,103)(0,-12)
      \put(0,0){\includegraphics{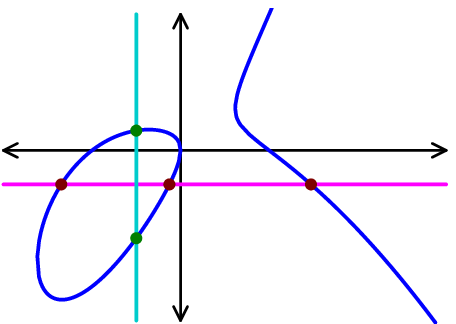}}
      \put(111,28){$\bL^{01}$}
      \put(20,78){$\bL^{10}$}
      \put(18,-12){Witness Collection}
   \end{picture}
\]
\end{example}

\subsection{Refining}
Suppose that $k=2$, so that $\ndot=(n_1,n_2)$ and set $n:=n_1+n_2$.
Let $Y\subset\PP^n$ be an irreducible variety of dimension $e$ and degree $d$.
Let $\Lambda$ be a linear polynomial that does not vanish  identically on $Y$ and set
$\defcolor{Y_{\aff}}:=Y\smallsetminus\calV(\Lambda)$, 
which is an affine variety in the affine patch $\CC^n=\PP^n\smallsetminus\calV(\Lambda)$.
For the splitting $\CC^n=\CC^{n_1}\times\CC^{n_2}$, 
let \defcolor{$Y_{\ndot}$} be the closure of $Y_{\aff}$ in the
compactification $\PP^{n_1}\times\PP^{n_2}$ of $\CC^{n_1}\times\CC^{n_2}$.

\begin{proposition}\label{Prop:perestroika}
  The multiprojective variety $Y_{\ndot}$ and multiaffine variety $Y_{\aff}$ are irreducible and have dimension $e$.
  For any $(e_1,e_2)\in[\ndot]$ with $e=e_1+e_2$, the $(e_1,e_2)$-multidegree of $Y_{\ndot}$ (and also of $Y_{\aff}$) is at
  most the degree of $Y$.
\end{proposition}

This agrees with Example~\ref{ex:cubicPerestroika}, where the size of each set in the witness collection was bounded above
by the degree of the plane curve.

\begin{proof}
 As $Y_{\aff}$ is a nonempty open subset of the irreducible variety $Y$, it is irreducible and of the same dimension.
 The same arguments imply that $Y_{\ndot}$ is irreducible of dimension $e$.
 A general multilinear section $\calV(\bL^{(e_1,e_2)})\cap Y_{\ndot}$ will be a subset of $Y_{\aff}$.
 In the affine space $\CC^n=\CC^{n_1}\times\CC^{n_2}$, the variety  $\calV(\bL^{(e_1,e_2)})$ is a (non-general) linear
 subspace of codimension~$e$, and thus $\calV(\bL^{(e_1,e_2)})\cap Y_{\aff}$ consists of at most $\deg(Y)$ points.
\end{proof}

This gives a homotopy algorithm for computing witness collections under a refinement of a
coordinate partition.
Let $Y\subset\CC^n$ be an equidimensional affine variety of dimension~$e$ and degree $d$, given as a union of components of
a variety $\calV(F)$ and let $\CC^n=\CC^{n_1}\times\CC^{n_2}$ be a splitting of $\CC^n$ with
$\bfy=(\bfy_1,\bfy_2)$ the corresponding partition of variables for $\CC^n$.

Suppose that $Y\subset\CC^n$ is represented by a witness set $(F,L^e,Y\cap\calV(L^e))$  where
$L^e\subset\CC[\bfy]$ consists of $e$ general affine forms.
Let $\bL^{(e_1,e_2)}$ be $e$ affine forms with $e_1$ from $\CC[\bfy_1]$ and  $e_2$ from $\CC[\bfy_2]$, but
otherwise general.
Then, $Y\cap\calV(\bL^{(e_1,e_2)})$ is a $(e_1,e_2)$-witness point set
for the multiaffine variety $Y\subset\CC^{n_1}\times\CC^{n_2}$.
The system 
 \begin{equation}\label{Eq:Refining_homotopy}
  \defcolor{H(t)}\ :=\ (F\,,\, t L^e + (1{-}t)\bL^{(e_1,e_2)})
 \end{equation}
is a homotopy that connects the solutions $Y\cap\calV(L^e)$ of the
start system $H(1)$ to solutions $Y\cap\calV(\bL^{(e_1,e_2)})$ of the target system  $H(0)$.

\begin{alg}[Transforming witness sets under refinement]\label{Alg:Refinement}
  \mbox{\ }\newline
  {\bf Input:} A witness set $(F,L^e,Y\cap\calV(L^e))$ for an equidimensional affine variety $Y\subset\CC^n$ of dimension
  $e$, a splitting  $\CC^n=\CC^{n_1}\times\CC^{n_2}$, and integers $0\leq e_1,e_2$ with $e_1+e_2=e$.\newline
  {\bf Output:} An $(e_1,e_2)$-witness point set for the multiaffine variety $Y\subset\CC^{n_1}\times\CC^{n_2}$. \newline
  {\bf Do:} Form the homotopy~\eqref{Eq:Refining_homotopy} and follow the points of $Y\cap\calV(L^e)$ along $H$ from $t=1$
  to $t=0$, keeping those whose paths are bounded near $t=0$.
\end{alg}
  
 Executing Algorithm~\ref{Alg:Refinement} for each $(e_1,e_2)\in\Dim(Y)$ computes the witness point sets for the full
witness collection of the multiaffine variety $Y\subset\CC^{n_1}\times\CC^{n_2}$.
 In Example~\ref{ex:cubicPerestroika} Algorithm~\ref{Alg:Refinement} amounts to rotating the line $L$ in the middle picture
 to either a horizontal or a vertical line.

\begin{proof}[Proof of correctness] 
 Since $L^e$ is general, the intersection $Y\cap\calV(L^e)$ is transverse and consists of $d=\deg(Y)$ points.
 Thus, for general $t$, the intersection \mbox{$Y\cap\calV(t L^e + (1-t)\bL^{(e_1,e_2)})$} is also transverse and consists of
 $d$  points, and so~\eqref{Eq:Refining_homotopy} is a homotopy. 
 As the affine forms in $\bL^{(e_1,e_2)}$ are general given their variables, the intersection 
 $Y\cap\calV(\bL^{(e_1,e_2)})$ is transverse and consists of $\Deg_Y(e_1,e_2)$ points.
 Thus $\Deg_Y(e_1,e_2)$ paths in the homotopy end at the points of $Y\cap\calV(\bL^{(e_1,e_2)})$ and  
 $d-\Deg_Y(e_1,e_2)$ paths diverge as $t$ approaches~0.
\end{proof}

\begin{rem}\label{R:SplittingManyFactors}
 Suppose that $Y\subset\CC^{n_1}\times\dotsb\times\CC^{n_k}$ is a multiaffine variety with $k>1$ and that 
 $\CC^{n_i}=\CC^{n'_i}\times\CC^{n''_i}$ is a refinement splitting the $i$th factor $\CC^{n_i}$.
 We may use the ideas in Algorithm~\ref{Alg:Refinement} to transform an 
 $\bfe$-witness point set for $Y$ into one
 for this refinement.
 Given an $\bfe$-witness point 
 set $Y\cap\calV(\bL^{\bfe})$, we wish to compute a 
 $\bfe'$-witness point set for the refinement, where the component $e_i$ of $\bfe$ is split into $e_i'+e_i''$ in $\bfe'$.
 For this, let $\bL^{(e_i',e_i'')}_i$ be $e_i$ general affine forms with
 $e'_i$ in $\CC[\bfy_{i'}]$ and $e''_i$ in $\CC[\bfy_{i''}]$, where $\bfy_i=(\bfy_{i'},\bfy_{i''})$  is the
 corresponding split of the variable group  $\bfy_i$.
 Replacing the $e_i$ affine forms of $L_i\subset\CC[\bfy_{i}]$ in~$\bL^{\bfe}$ by the convex combination
 $tL_i+(1-t)\bL^{(e_i',e_i'')}_i$ gives a homotopy as in Algorithm~\ref{Alg:Refinement} that transforms
 $Y\cap\calV(\bL^{\bfe})$ into $Y\cap\calV(\bL^{\bfe'})$.
\end{rem} 

\subsection{Coarsening}\label{SS:Coarsening}

Suppose that  $\ndot=(n_1,n_2)$ and set $\defcolor{n}:=n_1+n_2$.
Let $Y_{\ndot}\subset\PP^{n_1}\times\PP^{n_2}$ be an irreducible multiprojective variety of intrinsic dimension $e$.
For each $i=1,2$, let $\Lambda_i\in\CC[\bfx^{(i)}]$ be a general linear polynomial.
Then
\[
  \defcolor{Y_{\aff}}\ :=\
  Y_{\ndot}\smallsetminus \calV(\Lambda_1\cdot\Lambda_2)\ \subset\
  \PP^{n_1}\times\PP^{n_2}\smallsetminus \calV(\Lambda_1\cdot\Lambda_2)\ \simeq\
  \CC^{n_1}\times\CC^{n_2}
\]
is a multiaffine variety with the same multidimension and multidegree as $Y_{\ndot}$.
Regarding $\CC^{n_1}\times\CC^{n_2}=\CC^n$ as an affine patch in $\PP^n$, let \defcolor{$Y$} be the closure of
$Y_{\aff}$ in $\PP^n$.
We investigate how to transform a witness collection for $Y_{\ndot}$ into a witness set for $Y$. 
In particular, we describe Algorithm~\ref{Alg:Coarsening}, which 
transforms witness sets under coarsening.
In this algorithm we first 
construct a witness set for the Segre embedding $\sigma(Y)$ from witness points of $Y_{\ndot}$
and then degenerate this witness set into another witness set of $\sigma(Y)$ whose pullback is a witness set for $Y$.
In fact, all steps of Algorithm~\ref{Alg:Coarsening} operate in local coordinates for
$\PP^{n_1}\times\PP^{n_2}$,
not in the ambient space for the Segre map.

The multiprojective space $\PP^{n_1}\times\PP^{n_2}$ is a projective variety under the Segre map
\[
  \sigma\ \colon\ \PP^{n_1}\times\PP^{n_2}\ \hookrightarrow\
  \PP^{n_1n_2+n_1+n_2}\ =:\ \defcolor{\PP^N}\,.
\]
A linear polynomial on $\PP^N$ pulls back to a bilinear form $B$ on $\PP^{n_1}\times\PP^{n_2}$.
Writing $\PP^N$ as 
$\PP(\CC^{n_1+1}\otimes\CC^{n_2+1})=
   \PP(\Mat_{(n_1+1)\times(n_2+1)}(\CC))$, 
   a linear form $\ell$ on $\PP^N$ corresponds to a matrix $M$.
When $M$ has rank one,
the pullback  $\sigma^*(\ell)=\ell^{10}\ell^{01}$ is a product of linear polynomials, one in each set of variables.
Thus $\calV(t\ell^{10}\ell^{01}+(1{-}t)B)$ is a family (of hyperplane sections of $\sigma(\PP^{n_1}\times\PP^{n_2})$)
that transforms the union $Y_{\ndot}\cap(\calV(\ell^{10})\cup\calV(\ell^{01}))$ of multilinear sections into the bilinear
section $Y_{\ndot}\cap\calV(B)$.

Passing from  $\PP^{n_1}\times\PP^{n_2}$ to $\PP^n$ through affine patches, both $B$ and $\ell^{10}\ell^{01}$
remain bilinear forms.
Given a linear polynomial $\ell$ on $\PP^n$ and a choice $z_0$ of coordinate for the hyperplane at infinity, $z_0\ell$ is
another bilinear form whose variety in the affine patch $\PP^n\smallsetminus\calV(z_0)$ is the hyperplane $\calV(\ell)$.
Thus $tB+(1-t)z_0\ell$ or better $t\ell^{10}\ell^{01}+(1-t)z_0\ell$ is a family that may be used to transform the
union $Y_{\ndot}\cap(\calV(\ell^{10})\cup\calV(\ell^{01}))$ of sections of $Y_{\ndot}$ into
the union of the section $Y\cap\calV(\ell)$ with its part $Y\cap\calV(z_0)$ at infinity.

This may be used to transform a multilinear section $Y_{\ndot}\cap\calV(\bL^{(e_1,e_2)})$ into a subset of a
linear section $Y\cap\calV(L^e)$, but only if we work in an affine patch $\PP^n\smallsetminus\calV(z_0)$, as the
bilinear forms coming from linear polynomials in $L^e$ all have $z_0$ as a factor.
By the inequality among degree and multidegree in Proposition~\ref{Prop:perestroika},
we typically obtain a subset of $Y\cap\calV(L^e)$ (a partial witness set).

Let us describe a homotopy for this.
Let $\ell^{10}_1,\dotsc,\ell^{10}_e\in\CC[\bfy_1]$, $\ell^{01}_1,\dotsc,\ell^{01}_e\in\CC[\bfy_2]$, and
$\ell_1,\dotsc,\ell_e\in\CC[\bfy]$ be general affine forms.
Set $\defcolor{\bM}=(\ell^{10}_1\ell^{01}_1,\dotsc,\ell^{10}_e\ell^{01}_e)$ and \mbox{$\defcolor{L^e}=(\ell_1,\dotsc,\ell_e)$}.
Form the homotopy
 \begin{equation}\label{Eq:Coarsening_homotopy}
  H(t)\ :=\  ( F\,,\, t\bM\ +\ (1-t)L^e  \,)\,.
 \end{equation}
 We describe the start points for $H(t)$ at $t = 1$.
 For a partition $S\sqcup T$ of $\{1,\dotsc,e\}$ with 
 \mbox{$(|S|,|T|)\in\Dim(Y_{\aff})$}, let
 $\defcolor{\bL^{S,T}} := \{\ell^{10}_i\mid i\in S\} \cup \{\ell^{01}_j\mid j\in T\}$,
 a subset of linear forms involved above.
 Then, $(F,t\bL^{(|S|,|T|)}+(1-t)\bL^{S,T})$ is a homotopy transforming the witness point set 
 $Y_{\aff}\cap \calV(\bL^{(|S|,|T|)})$ into the multilinear section $\defcolor{W_{S,T}}:=Y_{\aff}\cap \calV(\bL^{S,T})$,
 which is a transverse intersection as the affine forms are general.
 Let \defcolor{$W$} be the disjoint union of all the $W_{S,T}$.  These are disjoint as the
 affine forms are general. 
 
\begin{alg}[Transforming witness sets under coarsening]\label{Alg:Coarsening}
  \mbox{\ }\newline
  {\bf Input:} A witness collection $\{(F,\bL^{\bfe},Y\cap\calV(\bL^{\bfe}))\mid \bfe\in\Dim(Y)\}$ for an equidimensional
  multiaffine variety $Y_{\aff}\subset\CC^{n_1}\times\CC^{n_2}$.\newline
  {\bf Output:} A witness point set $Y^\circ\cap\calV(L^e)$ for the affine variety
  $Y^\circ\subset\CC^{n}=\PP^n\smallsetminus\calV(z_0)$. \newline  
  {\bf Do:}
  Recall from above that $W$ is the union of all the $W_{S,T}$.
  Compute the points of $W$ and use the homotopy~\eqref{Eq:Coarsening_homotopy} to follow the points of $W$ along~$H$ from
  $t=1$ to $t=0$, keeping those whose paths are bounded near $t=0$.
\end{alg}

\begin{proof}[Proof of correctness]
  Observe that in  $\PP^{n_1}\times\PP^{n_2}$ we have 
 \[
   \calV(\bM)\ =\ \calV(\ell^{10}_1\ell^{01}_1,\dotsc,\ell^{10}_e\ell^{01}_e)\ =\ 
   \bigcup_{S\sqcup T=\{1,\dotsc,e\}} \calV(\bL^{S,T})\ .
 \]
 As $Y_{\ndot}\cap\calV(L^{S,T})=W_{S,T}$, we have 
  \[
    Y_{\ndot}\cap \calV(\bM)\ =\ 
     \bigcup_{S\sqcup T=\{1,\dotsc,e\}} Y_{\ndot}\cap \calV(\bL^{S,T})=\ 
     \bigsqcup_{S\sqcup T=\{1,\dotsc,e\}} W_{S,T}\ =\ W \ .
 \]
 By~\eqref{Eq:Segre-degree}, we have
 \begin{equation}\label{Eq:SegreDeg}
    \deg_{\PP^N}(\sigma(Y_{\ndot}))\ =\ \sum_{e_1+e_2=e}\binom{e}{e_1} \Deg_{Y_\ndot}(e_1,e_2)\ =\ |W|\,,
 \end{equation}
 as $|W_{S,T}|= \Deg_{Y_\ndot}(|S|,|T|)$.
 Thus $ Y_{\ndot}\cap \calV(\bM)$ is a transverse intersection consisting of
 $\defcolor{\delta}:=\deg_{\PP^N}(\sigma(Y_{\ndot}))$ points.
 Since $W\subset Y_{\aff}$, for general $t$ the intersection
 \begin{equation}\label{Eq:htpySection}
   Y_{\aff}\cap \calV(t \bM\ +\  (1-t)L^e)
 \end{equation}
 is also transverse and consists of $\delta$ points.
 Thus,~\eqref{Eq:Coarsening_homotopy} is a homotopy.

 Consider the variety in $\PP^n\times\CC_t$ defined by
 \begin{equation}\label{Eq:projSection}
   (Y\times\CC_t)\ \bigcap\ 
   \calV(t \bM\ +\  (1-t)(z_0\ell_1,\dotsc,z_0\ell_e))\,.
 \end{equation}
 (Note the homogenizing variable $z_0$.)
 Since~\eqref{Eq:htpySection} is transverse and consists of $\delta$ points for general $t$, the components
 of~\eqref{Eq:projSection} that map onto $\CC_t$ form a curve $C$ whose general fiber over $t$ is $\delta$ points.
 Restricting $C$ to an arc in $\CC_t$ with endpoints $\{0,1\}$ gives $\delta$ arcs that start (when $t=1$) at the points of
 $W$ and end (at $t=0$) in
 \[
   Y\cap\calV(z_0\ell_1,\dotsc,z_0\ell_e)\ =\
   Y\cap\calV(L^e)\;\bigcup\;
   Y\cap\calV(z_0)\,.
 \]
 As the affine forms $\ell_1,\dotsc,\ell_e$ are general, $Y\cap\calV(L^e)\subset Y^\circ$ is
 transverse and consists of $d=\deg_{\PP^n}(Y)$ points.
 Thus $d$ paths in the homotopy end at the points of $Y^\circ\cap\calV(L^e)$ and 
 $\delta-d$ paths diverge as $t$ approaches 0. 
\end{proof}

Algorithm~\ref{Alg:Coarsening} may use up to $2^e$ witness sets and tracks
$\delta=\deg_{\PP^N}(\sigma(Y_{\ndot}))$ as in~\eqref{Eq:SegreDeg} paths.
While $\delta$ is typically enormous, when $Y_{\ndot}$ is a curve, 
$\delta=\Deg_{Y_{\ndot}}\,(1,0)+\Deg_{Y_{\ndot}}\,(0,1)$, which is the cardinality of the witness collection for $Y_{\ndot}$.
In Example~\ref{ex:cubicPerestroika}, Algorithm~\ref{Alg:Coarsening} starts with five points on the intersection of the
cubic curve with the horizontal and vertical lines at right, passes through a family of hyperbolas, and ends at the
intersection of the cubic curve with the line $L$ in the middle, with two paths diverging to infinity.

\begin{rem}\label{R:CoarsenManyFactors}
 Suppose that $Y\subset\CC^{n_1}\times\dotsb\times\CC^{n_k}$ is a multiaffine variety with $k\geq 2$.
 Let $Y'\subset\CC^{n_1+n_2}\times\CC^{n_3}\times\dotsb\times\CC^{n_k}$ be the variety $Y$ with the multiaffine
 structure induced by merging the first two factors.
 As in Remark~\ref{R:SplittingManyFactors}, Algorithm~\ref{Alg:Coarsening} may be used to transform a witness collection
 for $Y$ into one for $Y'$.

 For each $i=1,\dotsc,k$, let $\bfy_i$ be $n_i$ indeterminates---these are the indeterminates for
 $\CC^{n_1}\times\dotsb\times\CC^{n_k}$.
 Write $\bfx:=(\bfy_1,\bfy_2)$.
 Then the partition of the indeterminates for $\CC^{n_1+n_2}\times\CC^{n_3}\times\dotsb\times\CC^{n_k}$
 is $(\bfx,\bfy_3,\dotsc,\bfy_k)$.
 We explain how to compute an $\bfe'=(e,e_3,\dotsc,e_k)$-witness 
 point set for $Y'$ given a witness collection for
 $Y$.
 (The indexing in $\bfe'$ is intended.)

 An $\bfe'$-witness point set for $Y'$ is an intersection $Y'\cap\calV(\bL^{\bfe'})$, where
 $\bL^{\bfe'}=(L^e,L_3,\dotsc,L_k)$ with $L^e$ consisting of $e$ general affine forms
 $\ell_1,\dotsc,\ell_e\in\CC[\bfx]$ and for $i\geq 3$, $L_i$ consists of $e_i$ general affine forms
 in $\CC[\bfy_i]$.
 As in the discussion preceding Algorithm~\ref{Alg:Coarsening}, let $\ell^{10}_1,\dotsc,\ell^{10}_e\in\CC[\bfy_1]$
 and $\ell^{01}_1,\dotsc,\ell^{01}_e\in\CC[\bfy_2]$ be general affine forms and set
 $\bM$ to be $(\ell^{10}_1\ell^{01}_1,\dotsc,\ell^{10}_e\ell^{01}_e)$.
 Following the same discussion,  for each $S\sqcup T=\{1,\dotsc,e\}$ construct $\bL^{S,T}$, substitute this for $L^e$ in
 $\bL^{\bfe'}$, use the $(|S|,|T|,e_3,\dotsc,e_k)$th witness set for $Y$ to compute $W_{S,T}$, and set $W$ to be the union
 of the $W_{S,T}$. 

 Let $\bL(t)$ be the convex combination $(t\bM+(1-t)L^e\,,\,L_3,\dotsc,L_k)$.
 Then, as in Algorithm~\ref{Alg:Refinement}, we will have a homotopy that transforms the union of witness point sets
 $W=Y\cap\calV(\bM,L_3,\dotsc,L_k)$ into the $\bfe'$-witness point set $Y'\cap\calV(\bL^{\bfe'})$.\hfill$\diamond$
\end{rem} 

\begin{example}\label{Ex:coarseningOctahedron}
  Let us revisit Example~\ref{Ex:Octahedron}, which involved varieties  $\calV(f,g)$ and $\calV(f,h)$ in
  $\CC_x\times\CC_y\times\CC_z\times\CC_w$.
  Both have their multidimension the vertices of an octahedron.
  Of the many coarsenings, we consider four, merging either the last two factors, the first two, both the first and the
  last two, and finally the last three.
  Table~\ref{Ta:multidegrees} displays the multidegrees of the original varieties in
  $\CC_x\times\CC_y\times\CC_z\times\CC_w$ and after merging.\hfill$\diamond$
\end{example}

\begin{table}[htb]
  \caption{Coarsenings of $\calV(f,g)$ and $\calV(f,h)$ from Example~\ref{Ex:Octahedron}.}
  \label{Ta:multidegrees}
\begin{tabular}{|c||c|c|c|c|c|}\hline
   &$\CC_x{\times}\CC_y{\times}\CC_z{\times}\CC_w$&
    $\CC_x{\times}\CC_y{\times}\CC_{zw}^2$& $\CC_{xy}^2{\times}\CC_z{\times}\CC_w$&
    $\CC_{xy}^2{\times}\CC_{zw}^2$& $\CC_x{\times}\CC_{yzw}^3$\\\hline\hline
   \raisebox{27pt}{$\Dim$}&
    \begin{picture}(81,65)(4,0)
     \put(22, 6){\includegraphics{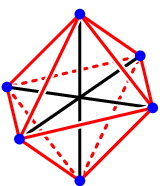}}
     \put(47,1){\small$1001$}     \put( 21,55){\small$0110$}
     \put( 5,11){\small$0011$}     \put( 62,47){\small$1100$}
     \put( 4,38){\small$0101$}     \put( 65,18){\small$1010$}
    \end{picture}
   &
   \raisebox{13pt}{\begin{picture}(61,40)(-7,-3)
     \put(0, 0){\includegraphics{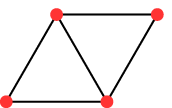}}
     \put( 8,30){\footnotesize$011$}  \put(38,30){\footnotesize$110$}
     \put(-6,-8){\footnotesize$002$}  \put(23,-8){\footnotesize$101$} 
   \end{picture}}
   &
    \raisebox{13pt}{\begin{picture}(61,40)(-7,-3)
     \put(0, 0){\includegraphics{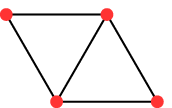}}
     \put(-7,30){\footnotesize$011$}  \put(23,30){\footnotesize$110$}
     \put( 8,-8){\footnotesize$101$}  \put(38,-8){\footnotesize$200$} 
   \end{picture}}
   &
  \raisebox{21pt}{\begin{picture}(66,20)(-3,0)
    \put(0,10){\includegraphics{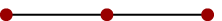}}
    \put(-4,0){\footnotesize$02$} \put(26,0){\footnotesize$11$}  \put(55,0){\footnotesize$20$} 
   \end{picture}}
   &
  \raisebox{21pt}{\begin{picture}(38,20)(-2,0)
   \put(0,10){\includegraphics{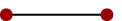}}
    \put(-4,0){\footnotesize$02$} \put(26,0){\footnotesize$11$} 
   \end{picture}}
  \\\hline
   \raisebox{25pt}{$\calV(f,g)$}&
    \begin{picture}(57,61)(16,2)
     \put(22, 6){\includegraphics{figures/Octahedron_sm}}
     \put(48, 2){\small$3$}     \put( 35,54){\small$4$}
     \put(19,16){\small$2$}     \put( 63,46){\small$4$}
     \put(16,33){\small$3$}     \put( 68,24){\small$4$}
    \end{picture}
   &
   \raisebox{11pt}{\begin{picture}(60,40)(-7,-3)
     \put(0, 0){\includegraphics{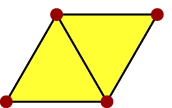}}
     \put( 8,28){\footnotesize$4$}  \put(48,28){\footnotesize$4$}
     \put(-7,-3){\footnotesize$2$}  \put(33,-3){\footnotesize$4$} 
   \end{picture}}
   &
    \raisebox{11pt}{\begin{picture}(60,40)(-7,-3)
     \put(0, 0){\includegraphics{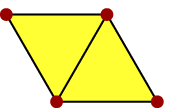}}
     \put(-7,28){\footnotesize$2$}  \put(33,28){\footnotesize$4$}
     \put( 8,-3){\footnotesize$3$}  \put(48,-3){\footnotesize$4$} 
   \end{picture}}
   &
  \raisebox{19pt}{\begin{picture}(65,20)(-2,0)
    \put(0,10){\includegraphics{figures/TwoSeg}}
    \put(-2,0){\footnotesize$2$} \put(28,0){\footnotesize$4$}  \put(57,0){\footnotesize$4$} 
   \end{picture}}
   &
  \raisebox{19pt}{\begin{picture}(38,20)(-2,0)
   \put(0,10){\includegraphics{figures/OneSeg}}
    \put(-2,0){\footnotesize$2$} \put(28,0){\footnotesize$4$} 
   \end{picture}}
  \\\hline
   \raisebox{25pt}{$\calV(f,h)$}&
    \begin{picture}(57,61)(16,2)
     \put(22, 6){\includegraphics{figures/Octahedron_sm}}
     \put(48, 2){\small$5$}     \put( 35,54){\small$5$}
     \put(19,16){\small$3$}     \put( 63,46){\small$7$}
     \put(16,33){\small$4$}     \put( 68,24){\small$6$}
    \end{picture}
   &
   \raisebox{11pt}{\begin{picture}(60,40)(-7,-3)
     \put(0, 0){\includegraphics{figures/Parallelogram}}
     \put( 8,28){\footnotesize$6$}  \put(48,28){\footnotesize$7$}
     \put(-7,-3){\footnotesize$3$}  \put(33,-3){\footnotesize$8$} 
   \end{picture}}
   &
    \raisebox{11pt}{\begin{picture}(60,40)(-7,-3)
     \put(0, 0){\includegraphics{figures/ParallelogramR}}
     \put(-7,28){\footnotesize$3$}  \put(33,28){\footnotesize$10$}
     \put( 8,-3){\footnotesize$8$}  \put(48,-3){\footnotesize$7$} 
   \end{picture}}
   &
  \raisebox{19pt}{\begin{picture}(65,20)(-2,0)
    \put(0,10){\includegraphics{figures/TwoSeg}}
    \put(-2,0){\footnotesize$3$} \put(28,0){\footnotesize$12$} \put(57,0){\footnotesize$7$} 
   \end{picture}}
   &
  \raisebox{19pt}{\begin{picture}(38,20)(-2,0)
   \put(0,10){\includegraphics{figures/OneSeg}}
    \put(-2,0){\footnotesize$7$} \put(28,0){\footnotesize$11$} 
   \end{picture}}
  \\\hline
\end{tabular}
\end{table}

\section{Slicing}\label{sec:slice}

While the dimension of an equidimensional affine or projective variety $X$ is reduced by~1 under a general
linear section (a \demph{slice}), its degree is preserved---if $\dim(X)\geq 1$.
Similarly, the (ir)reducibility of $X$ is preserved when $\dim(X)\geq 2$, by the classical Bertini Theorem.
Consequently decomposing a variety into irreducible components is reduced to the case of curves.

When $X$ is a multiprojective or multiaffine variety, information about its multidimension and multidegrees may be lost
under a general linear section, and its (ir)reducibility may not be preserved, even when $X$ has dimension at least 2.
However, this may be quantified and it leads to useful reductions. 
The subsequent reductions will be exploited in our algorithm for numerical irreducible decomposition in
Section~\ref{S:NID}. 

For $i=1,\dotsc,k$ let \defcolor{$\pi_{\{i\}}$} be the projection onto the $i$th factor in multiprojective or multiaffine 
space.
Let $\defcolor{\epsilon_i}\in\NN^k$ be the vector whose $i$th component is 1 and others are 0.

\begin{lem}\label{lem:Fslice}
  Let $X\subset\PP^\ndot$ (or $X\subset\CC^\ndot$) be an equidimensional multiprojective or multiaffine variety and suppose
  that $\ell$ is a general linear polynomial/affine form in the $i$th variable group.

\begin{enumerate}
\item  
     If $\dim\pi_{\{i\}}(X)=0$, then $X\cap\calV(\ell)=\emptyset$. 
     This is equivalent to $\bfe\in\Dim(X) \Rightarrow e_i=0$.
  
  \item If $\dim\pi_{\{i\}}(X)\geq 1$, then
   \begin{eqnarray*}
       \Dim(X\cap\calV(\ell)) &=& \{ \bfe-\epsilon_i \mid \bfe\in\Dim(X)\ \mbox{ and }\ e_i>0\}\,, \\
       \Deg_{X\cap\calV(\ell)}(\bfe) &=& \Deg_X(\bfe+\epsilon_i)\,.
   \end{eqnarray*}
    
  \item If $\dim\pi_{\{i\}}(X)\geq 2$, then $X$ is (ir)reducible if and only if $X\cap\calV(\ell)$ is (ir)reducible.
\end{enumerate}  
\end{lem}

\begin{proof}
  Statements (1) and (2) follow from the definitions given in Subsection~\ref{SS:Multiproj}, and (3) follows from the
  Bertini Theorem for maps to projective space~\cite[Thm.~6.3~(4)]{jouanolou}. 
\end{proof}
\newcommand{\bfm}{\mathbf{m}}

\begin{example}\label{Ex:slicing}
 Let $\ndot=(3,3,3)$ and consider a subvariety $X$ of $\PP^\ndot$ given by six general multihomogeneous polynomials, each of
 multidegree $(1,2,3)$.
 Then $X$ is irreducible and has intrinsic dimension 3.
 Following Remark~\ref{R:intTheory} and~\cite{Fulton}, its multidimension and multidegree are computed as follows.
 Its cohomology class in $\PP^\ndot$ is the normal form of $(s_1+2s_2+3s_3)^6$ in
 $\ZZ[s_1,s_2,s_3]/\langle s_1^4,s_2^4,s_3^4 \rangle$, which is 
  \begin{multline*}
   \qquad 160s_1^3s_2^3\,+\,720s_1^3s_2^2s_3\,+\,1440s_1^2s_2^3s_3\,+\,1080s_1^3s_2s_3^2\,+\,3240s_1^2s_2^2s_3^2\\
    \,+\,4320s_1s_2^3s_3^2\,+\,540s_1^3s_3^3\,+\,3240s_1^2s_2s_3^3\,+\,6480s_1s_2^2s_3^3\,+\,4320s_2^3s_3^3\ .\qquad
  \end{multline*}
  Its homology class is obtained by replacing $s_1^{3-a}s_2^{3-b}s_3^{3-c}$ by
  $\bfT^{abc}$. 
  We display its multidimension and multidegree below.
  (The central point in $\Dim(X)$ is $111$.)
 \begin{equation}\label{Eq:JoseExample}
 \raisebox{-33pt}{\begin{picture}(85,73)(-9,0)
    \put(-1.5,9.5){\includegraphics{figures/DR0}}
    \put(14,67){\small$030$} 
    \put( 3,47){\small$021$}\put( 47,47){\small$120$}
    \put(-9,27){\small$012$}                         \put( 59,27){\small$210$}
    \put(-9, 0){\small$003$}\put( 14, 0){\small$102$}\put( 36, 0){\small$201$}\put(58, 0){\small$300$}
   \put(23,33){{\color{white}\rule{20pt}{10pt}}}  \put(25,34.5){\small$111$}
  \end{picture}}
  \qquad
 \raisebox{-41pt}{\begin{picture}(107,90)(-9,0)
    \put(-1.5,9.5){\includegraphics{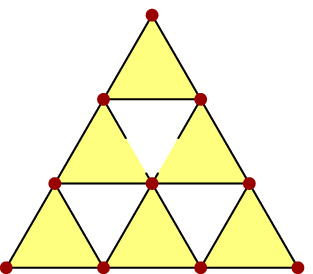}}
    \put(23,82){\small$540$} 
    \put(  3,58){\small$1080$}\put( 60,58){\small$3240$}
    \put(-5,33){\small$720$}    \put( 31,40){\small$3240$}   \put(73.5,33){\small$6480$}
    \put(-9, 0){\small$160$}\put( 17, 0){\small$1440$}\put( 45.5, 0){\small$4320$}\put(74, 0){\small$4320$}
  \end{picture}}
 \end{equation}
Slicing with a general linear polynomial $\calV(\ell)$ in the $i$th variable group gives a variety of dimension two in a
product $\PP^2\times\PP^3\times\PP^3$ (permuted so that $\PP^2$ is the $i$th factor) with multidimension and multidegrees
as shown below.
\[
  \begin{picture}(71,72)(-16,-11)
    \put(-3,9.5){\includegraphics{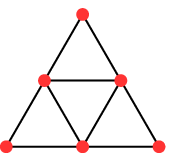}}
    \put(12,53){\small$020$}
    \put(-12,27){\small$011$} \put(36,27){\small$110$}
    \put(-16, 0){\small$002$}\put( 12, 0){\small$101$}\put( 39, 0){\small$200$}
  \end{picture}
   \qquad
  \begin{picture}(76,74)(-16,-12)
    \put(-3,9.5){\includegraphics{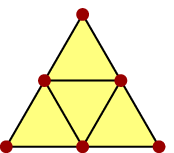}}
    \put(10,53){\small$1080$}
    \put(-11,27){\small$720$} \put(36,27){\small$3240$}
    \put(-16, 0){\small$160$}\put(10, 0){\small$1440$}\put( 39, 0){\small$4320$}
    \put(9,-12){\small$i=1$}
  \end{picture}
   \qquad
  \begin{picture}(77,74)(-16,-12)
    \put(-3,9.5){\includegraphics{figures/D2}}
    \put(12,53){\small$540$}
    \put(-16,27){\small$1080$} \put(36,27){\small$3240$}
    \put(-16, 0){\small$720$}\put( 10, 0){\small$3240$}\put( 39, 0){\small$6480$}
    \put(9,-12){\small$i=2$}
  \end{picture}
   \qquad
  \begin{picture}(80,74)(-20,-12)
    \put(-3,9.5){\includegraphics{figures/D2}}
    \put(10,53){\small$3240$}
    \put(-16,27){\small$3240$} \put(36,27){\small$6480$}
    \put(-20, 0){\small$1440$}\put( 10, 0){\small$4320$}\put( 39, 0){\small$4320$}
    \put( 9,-12){\small$i=3$}
  \end{picture}
\]
 Slicing with another general linear polynomial gives an irreducible curve in either $\PP^1\times\PP^3\times\PP^3$ or
 $\PP^2\times\PP^2\times\PP^3$ (with possibly permuted factors) of multidimension $\{001, 010, 100\}$, and
 multidegrees corresponding to one of the six upright shaded triangles in~\eqref{Eq:JoseExample}.\hfill$\diamond$
\end{example}

A consequence of Lemma~\ref{lem:Fslice} is that obtaining a witness collection for a linear slice $X\cap\calV(\ell)$
from one for $X$ is a matter of bookkeeping.
We recall the definitions from Subsection~\ref{SS:Multiproj}.
Let $X$ be an equidimensional union of components of $\calV(F)$.
Choose general linear polynomials $\ell_{i,1},\dotsc,\ell_{i,n_i}\in\CC[\bfx_i]$ for each $i\in \{1,\dots,k\}$, and for
$\bfe\in[\ndot]$, set $\defcolor{L_i^{e_i}}:=(\ell_{i,1},\dotsc,\ell_{i,e_i})$ and
$\defcolor{\bL^\bfe}:=(L_1^{e_1},\dotsc,L_k^{e_k})$. 
Then, for $\bfe\in\Dim(X)$, the $\bfe$-witness set is $(F,\bL^\bfe,W_{\bfe})$, where
$\defcolor{W_\bfe}:=X\cap\calV(\bL^{\bfe})$. 
We use the same notation for a multiaffine variety $X\subset\CC^{\ndot}$.

\begin{alg}[Witness collection of a slice]\label{alg:mSlice}
  \mbox{\ }\newline
  {\bf Input:} A witness collection $\{(F,\bL^{\bfe},W_\bfe)\}$ for an equidimensional multiprojective or multiaffine variety
  $X\subset\PP^\ndot$ (or $X\subset\CC^\ndot$) and an index $i\in\{1,\dots,k\}$.
  \newline 
  {\bf Output:} A witness collection for $X\cap \calV(\ell_{i,1})$.
    \newline
  {\bf Do:} Return $\{(F\cup\{\ell_{i,1}\},\bL^{\bfe}\smallsetminus\{\ell_{i,1}\},W_\bfe) \mid \bfe\in\Dim(X)\mbox{ and }e_i>0\}$.
\end{alg}

\begin{proof}[Proof of correctness]
  By Lemma~\ref{lem:Fslice}, $\Dim(X\cap \calV(\ell_{i,1}))=\{\bfe-\epsilon_i \mid \bfe\in\Dim(X)\mbox{ and }e_i>0\}$,
  as $\ell_{i,1}$ is general.
  Moreover, $W_{\bfe}=X\cap \calV(\bL^{\bfe})=X\cap \calV(\ell_{i,1}) \cap\calV(\bL^{\bfe}\smallsetminus\{\ell_{i,1}\})$,
  so that~$W_{\bfe}$ is both an $\bfe$-witness point 
  set for $X$  and an $(\bfe-\epsilon_i)$-witness point set for $X\cap \calV(\ell_{i,1})$.
\end{proof}  

\begin{rem}\label{R:IterateSlicing}
 As indicated in Example~\ref{Ex:slicing}, both Lemma~\ref{lem:Fslice} and Algorithm~\ref{alg:mSlice} may be applied in
 succession to a variety and its witness collection.
 If the projection of the variety to the $i$th factor has dimension at least two, this preserves the irreducible
 components.
 The choice of slice affects the size of the output.
 In Example~\ref{Ex:slicing}, slicing twice with a linear polynomial in $\bfx_1$ gives a witness collection with
 $160+720+1440=2320$ points, while slicing twice with a linear polynomial in $\bfx_3$ yields $4320+4320+6480=15120$ points. 
\end{rem}  

\section{Numerical decompositions of algebraic varieties}
\label{S:NID}

\subsection{Affine and projective varieties}\label{SS:NIDaffine}
\newcommand{\aComp}{Y}
Any variety has a unique (irredundant) decomposition into irreducible components.
For subvarieties of $\CC^n$ or $\PP^n$, a 
numerical irreducible decomposition 
mirrors the irreducible decomposition by 
producing a formal union of witness sets, one for each irreducible component.
As described in Subsection~\ref{SS:witnessSets},
we summarize a well-known approach for computing
a numerical irreducible decomposition 
for equidimensional varieties in $\CC^n$.

\begin{alg}[Equidimensional numerical irreducible decomposition in $\CC^n$]\label{alg:NIDaffine}
  \mbox{\ }\newline
       {\bf Input:} A witness set $(F,L,W)$ for equidimensional $X\subset\calV(F)$.
       \newline 
      {\bf Output:} A numerical irreducible decomposition of $X$.
      \newline
      {\bf Do:} Perform monodromy loops to partition the witness point set
      $W$ into subsets of points $P_1\sqcup\cdots\sqcup P_s$       
      where all points in each $P_i$  lie on the same irreducible component.
      Repeat until the trace test confirms that each $P_i$ is a witness point set for some irreducible component yielding
      the numerical irreducible decomposition $\sqcup_i (F,L,P_i)$.
\end{alg}

For varieties that are not equidimensional, one simply performs a numerical irreducible decomposition on each of its
equidimensional components.  

\subsection{Multiprojective varieties}\label{SS:NIDmulti}

For a numerical irreducible decomposition of a multiprojective variety $\mathcal{V}(F)$,
it makes sense to set a goal to partition an {\it arbitrary} set of general points according to membership in the irreducible components.
Algorithm~\ref{alg:NID} (below) still applies to points in witness collections and can be modified to look similar to Algorithm~\ref{alg:equi} in a special case, which is explained in Remark~\ref{rem:old-to-new-WC}.
 
The underlying idea is to systematically loop
through the given points and determine if a point lies
on a previously computed irreducible component.
If not, information about this new irreducible component
must be computed and it is then added to the list of known components.
At a minimum, a membership test for this new irreducible
component must be developed.  
To that end, one uses the given point $p$ to compute the (multi)dimension of the corresponding irreducible component. 
This determines a sequence of slices that preserve irreducibility
(Lemma~\ref{lem:Fslice}) and coarsenings~(Subsection~\ref{SS:Coarsening}) 
which can be used to produce a system of linear polynomials~$L$ 
vanishing at~$p$ such that the irreducible component becomes an irreducible affine
curve $C_L \subset \mathcal{V}(F)\cap\mathcal{V}(L)$ containing~$p$. 
Then, a complete witness point set for $C_L$ can be constructed from $p$ via monodromy
(Algorithm~\ref{alg:completion}), with the trace test used as a stopping criterion.

Given another point $q$ from the given set of general points,
one first checks if $p$ and $q$ have the same (multi)dimension.
If so, then one produces a system $L'$ of similar structure
to $L$ but vanishing at $q$.  After computing a witness point set of 
the corresponding $C_{L'}$ from~$C_L$, 
testing membership (Remark~\ref{r:move-witness})
of $q$ in $C_{L'}$ is equivalent to determining 
if $p$ and $q$ lie on the same irreducible component.
Note that genericity assumptions on $p$ and~$q$ are needed here to avoid 
losing transversality as in \cite[Ex.~3.2]{HR15}.

\begin{alg}[Numerical irreducible decomposition in $\CC^\ndot$]\label{alg:NID}
\mbox{\ }\newline
{\bf Input:} A finite set $W\subset \calV(F) \subset\CC^{\ndot}$ of general smooth points.
\newline 
{\bf Output:} A partition of $W$ into sets corresponding to irreducible components. \newline 
{\bf Do:} Make use of the developed toolkit to represent irreducible components of $\calV(F)$ containing points $W$ with curves in an affine space and use this representation to sort the points into the respective components as follows:

\begin{itemize}
  \item Initialize the numerical irreducible decomposition $N$ to be the empty set. 
  \item While $W$ is nonempty:
  \begin{enumerate}
    \item Select a point $p\in W$, and let $X_p$ denote the irreducible component of $\calV(F)$ containing $p$.  
    \item Determine $\Dim(X_p)$ using Algorithm~\ref{A:local-dimension}.
    \item\label{item:m} Let $\bfm$ denote a coordinatewise maximal integer vector in $[\ndot]$ such that $\bL^\bfm$ is
        slice preserving the irredicubility of $X_p$.
    \item\label{item:bfe}  Choose an element $\bfe$ in $\Dim(X_p\cap \calV(\bL^\bfm)) \subset \{0,1\}^k$.
    \item\label{item:I}    
      Let $I = \{i_1,\dots,i_{|\bfe|}\}$ denote the positions of $\bfe$ with a nonzero entry indexed such that
      $ j= \dim \pi_{\{i_1,\dots, i_j\}}(X_p)$.
    \item  For $j\in \{ 1,\dots,|\bfe|-1\}$,
    let $\ell_j$ be a general linear polynomial in $\bfx_{i_1},\dots,\bfx_{i_{j+1}}$.
    \item\label{item:MT} Let $L := \bL^\bfe \cup \{\ell_1,\dots,\ell_{|\bfe|-1}\}$.    
    Use Algorithm~\ref{alg:completion} to compute a witness set  for 
    $C_L :=  X_p\cap \calV(L)$, 
    thereby deriving a membership test for $X_p$. 
    \item
    Let $M_p$ denote the points of $W$ which are members of $X_p$.
    \item
    Replace $W$ with $W\setminus M_p$ and append to $N$ the set $M_p$.    
  \end{enumerate}
  \item Return the numerical irreducible decomposition $N$ of $W$.
\end{itemize}

\end{alg}

\begin{proof}[Proof of correctness]
Steps \eqref{item:m}, \eqref{item:bfe},  and \eqref{item:MT} are the only steps needing
further justification.  
The existence of $\bfm$ in Step~\eqref{item:m} is a consequence of iterating Lemma~\ref{lem:Fslice}.
Furthermore, for $i\in\{1,\dots, k\}$ the projection $\pi_{\{i\}}(X\cap\calV(\bL^\bfm))$ is a point or a curve because
$\bfm$ is maximal.
Then it follows $\Dim(X_p\cap \calV(\bL^\bfm)) \subset \{0,1\}^k$ as stated in Step~\eqref{item:bfe}.
Finally, by Remark~\ref{r:move-witness-muttiprojective},
a witness set $(F\cup L,\ell, S )$ for $C_L$ yields a membership test for $q$ in the variety $X_p$ using the homotopy 
$H(t)=(F,\,L-(1-t)L(q),\, \ell-(1-t)\ell(q))$ with start points $S$.
\end{proof}

Theorem~\ref{thm:Product} can be used to simplify computation for the points that belong to components that are
Cartesian products by an obvious divide-and-conquer procedure.  

\begin{rem}\label{rem:old-to-new-WC}
 Given a complete witness collection for an equidimensional variety $X\subset\calV(F)\subset\CC^{\ndot}$, one can
 apply monodromy as in Algorithm~\ref{alg:NIDaffine} as a heuristic
 for numerical irreducible decomposition. 
 However, without slicing and coarsening, the trace test
 cannot be used to ensure the completion of such an algorithm.

 Coarsening changes the geometry, rendering the prior witness collection irrelevant.
 One can create a new witness collection using the old one in the fashion of Algorithm~\ref{Alg:Coarsening} 
 to avoid using monodromy to reconstruct witness points in the hope of reducing the computational cost.
 In some cases, the completeness of the new witness collection is guaranteed.
 For a curve in a product of two projective spaces considered in detail in~\cite{traceLRS},
 the original witness collection is linked to the new witness set with an optimal (one-to-one) homotopy. 
  Hence, a decomposition via Algorithm~\ref{alg:NIDaffine} on the new witness set induces a decomposition on the
  original~witness~collection. 

  In general, however, a witness set produced by Algorithm~\ref{Alg:Coarsening} is incomplete and thus, this one-to-one
  correspondence is lost.
  One can choose to complete the witness set and decompose the result yielding a decomposition of the original witness
  collection.  
  Whether this approach has an advantage over the method outlined in the beginning of this subsection depends on the number
  and nature of coarsenings taken. 
\end{rem}

\begin{example}\label{Ex:OctahedronNID}
Consider $Y:=\calV(f,h)\subset\CC_x\times\CC_y\times\CC_z\times\CC_w$
from Example~\ref{Ex:Octahedron}.  
We discuss three of the ways to reduce to an affine curve preserving irreducibility.  

One could simply coarsen to $\CC^4$ and then intersect with a general hyperplane in $\CC^4$.
This yields an irreducible curve of degree $15$ showing that $Y$ is irreducible.

Another option is to coarsen $\CC_y\times\CC_z\times\CC_w$ to $\CC^3$,
intersect with a general hyperplane in this~$\CC^3$, and then coarsen $\CC_x\times\CC^3$ to $\CC^4$.
This also yields an irreducible curve of degree~$15$.

A final option that we will consider is to coarsen $\CC_y\times\CC_z$ to $\CC^2$,
intersect with a general hyperplane in this~$\CC^2$, and then coarsen
$\CC_x\times\CC_w\times\CC^2$ to $\CC^4$.  
This yields an irreducible curve of degree~$12$.\hfill$\diamond$
\end{example}

As demonstrated in Example~\ref{Ex:OctahedronNID},
different reductions to affine curves can yield different 
degrees.  We leave it as a possible topic
for future research to consider finding combinations of 
coarsening, slicing, and potential factoring that result in 
the smallest degree.

\section{Fiber product examples}\label{sec:FiberProductEx}

Computing exceptional sets using fiber products \cite{SWFiberProducts} yields multihomogeneous systems.
We illustrate some of the tools from the toolkit on two examples: rulings of a hyperboloid and
exceptional planar pentads. 

\subsection{Rulings of a hyperboloid}\label{sec:hyperboloid}

Motivated by~\cite[\S~4]{SWFiberProducts}, we use fiber products to compute the two rulings of the hyperboloid $H\subset\CC^3$ defined by
\[
  h(\bfx)\ =\ x_1^2+x_2^2-x_3^2-1\,.
\]
\begin{figure}[!htb]
  \centering
  \includegraphics[height=120pt]{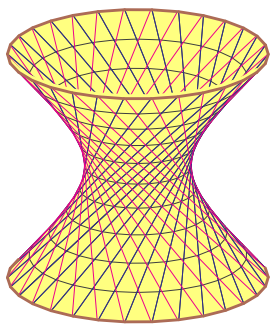}
  \caption{A hyperboloid with its two rulings.}\label{F:hyperboloid}
\end{figure}
Since a generic line meets $H$ in two points, a line which meets $H$ in three points must
be contained in $H$.  
For each $\blambda\in\CC_{\blambda}^4$, we associate 
a line  $\defcolor{L_{\blambda}}= \calV(\ell_{\blambda})\subset\CC^3$ where 
\[
     \ell_{\blambda}(\bfx)\ =\ \left(
           \lambda_1 x_1 + \lambda_2 x_2 - x_3 \,,\,
           \lambda_3 x_1 + \lambda_4 x_2 - 1\right)\ .
\]
%
%
Consider the following 
system
 on
$\CC_{\blambda}^4\times\CC_{\bfx_1}^3\times\CC_{\bfx_2}^3\times\CC_{\bfx_3}^3$,
\[
     F(\blambda,\bfx_1,\bfx_2,\bfx_3)\ =\ 
        \left(\ell_{\blambda}(\bfx_1)\,,\,
        \ell_{\blambda}(\bfx_2)\,,\,
        \ell_{\blambda}(\bfx_3)\,,\,
        h(\bfx_1)\,,\, h(\bfx_2)\,,\, h(\bfx_3) \right)\,.
\]
A ruling of $H$ corresponds to a four-dimensional 
irreducible component $X\subset\calV(F)$ such that there 
exists an irreducible curve $C\subset\CC_{\blambda}^4$ where
 \begin{equation}\label{eq:Xfiberproduct}
{
   X\ =\ \overline{\bigcup_{{\blambda}\in C} \{({\blambda},\bfx_1,\bfx_2,\bfx_3)
   : \bfx_i\in \calV(\ell_{\blambda})\}
   }
   }.
 \end{equation}
In particular, $\overline{\pi_1(X)} = C$
and $\overline{\pi_i(X)} = H$ for $i = 2,3,4$.

For $\bfe = (1,1,1,1)$, the witness point set $W_\bfe := \calV(F)\cap\calV(\bL^\bfe)$ consists of $16$ isolated points.  
The following uses our toolkit to determine the irreducible components corresponding to the rulings.

Using Algorithms~\ref{A:local-dimension} and~\ref{alg:equi},  we compute the dimensions of components  containing points of
$W_\bfe$ 
under different projections. 
The following table records the relevant information up to symmetry. 
 \[
   \begin{array}{c|c|c|c|c}  
     \hbox{$\#$~points~in~$W_\bfe$} & 
      \dim \overline{\pi_{\{1\}}(X)} & 
     \dim \overline{\pi_{\{2\}}(X)} & 
      \dim \overline{\pi_{\{1,2\}}(X)} & 
     \dim \overline{\pi_{\{1,2,3\}}(X)} \\\hline
         4 & 1 & 2 & 2 & 3 \\
         12 & 4 & 2 & 4 & 4
   \end{array}
 \]
The first column shows that the points of $W_\bfe$ lie on components of $\calV(F)$ having two
distinct multidimensions. 
 Let $W'_\bfe$ consist of the four points from the first row of this table.  
 For each $({\blambda},\bfx_1,\bfx_2,\bfx_3)\in W'_\bfe$, the last column implies that
 the fiber over $({\blambda},\bfx_{i_1},\bfx_{i_2})$ for distinct $i_1,i_2\in \{2,3,4\}$ is one-dimensional.
 The trace test shows that each irreducible component of the fiber is linear as expected from \eqref{eq:Xfiberproduct}. 
We note that the number of witness points 
when treating $F$ as system in $\CC^{13}$ 
is~$120$ and much larger than the number of points in the witness sets in the previous table.

Finally, we compute the irreducible components of $\overline{\pi_{\{1\}} (\mathcal{V}(F))}$.
 Using monodromy and the trace test in $\CC^4$, we partition $W'_\bfe$ into two sets of size two,
 each corresponding to a distinct ruling of the hyperboloid.
 The rulings correspond to the two irreducible curves
 \[
     \{(\lambda_1,\lambda_2,-\lambda_2,\lambda_1)\mid\lambda_1^2 + \lambda_2^2 = 1\}
      \qquad\mbox{and}\qquad
      \{(\lambda_1,\lambda_2,\lambda_2,-\lambda_1)\mid\lambda_1^2 + \lambda_2^2 = 1\}\,.
 \]      
{The symbolic expressions of these
curves are classically known, but
one can also recover them directly 
from the computed witness points via \cite{ExactnessRecovery}.}

\subsection{Exceptional planar pentads}\label{sec:planarpentad}

A planar pentad is a 3-RR mechanism constructed by connecting the vertices of two triangles
by three legs with revolute joints as shown in Figure~\ref{fig:planarpentad}.
\begin{figure}[htb]
  \begin{picture}(145,115)(-10,-5)
    \put(0,0){\includegraphics{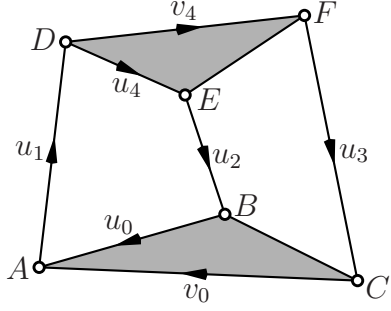}}
    \put(-1,90){$D$}  \put(52,103){$v_4$}   \put(106,100){$F$}
    \put(30,74){$u_4$}  \put(62,67){$E$}
    \put(-6.5,50){$u_1$}  \put(68,47){$u_2$}   \put(116,49){$u_3$}
    \put(27,22){$u_0$}  \put(76,27){$B$}
    \put(-10,3){$A$}  \put(57,-4){$v_0$}   \put(126,-4){$C$}
  \end{picture}
\caption{A planar pentad mechanism}\label{fig:planarpentad}
\end{figure}
We fix one of the triangles in the plane to remove the trivial motion of the entire mechanism.
A generic mechanism can be assembled in six different configurations, which is the degree of SE(2) as described
in~\cite[Table~1]{HSWexceptional}.
Since a generic planar pentad is \defcolor{rigid} (i.e., does not move), a planar pentad is \defcolor{exceptional} when it exhibits motion.

Using isotropic coordinates, the parameters in Figure~\ref{fig:planarpentad} are 
\[
   ([u_0,u_1,u_2,u_3,u_4,v_0,v_4], 
   [\ol{u}_0,\ol{u}_1,\ol{u}_2,\ol{u}_3,\ol{u}_4,\ol{v}_0,\ol{v}_4]  )\in\PP^6\times\PP^6
\]
 with the following assemblability restrictions:
 \[
     u_0+u_1+u_2+u_4 = \ol{u}_0+\ol{u}_1+\ol{u}_2+\ol{u}_4 = v_0+u_1+u_3+v_4 = 
     \ol{v}_0+\ol{u}_1+\ol{u}_3+\ol{v}_4 = 0.
  \]
A mechanism is \defcolor{degenerate}
if a parameter, $u_0-v_0$, 
$\ol{u}_0-\ol{v}_0$, 
$u_4-v_4$, or $\ol{u}_4-\ol{v}_4$
is zero.  Since nondegenerate mechanisms
are desired, we dehomogenize the parameter 
space by setting $v_0 = \ol{v}_0 = 1$.
Hence, the parameter space becomes
\[
  (\bfu,\ol\bfu)\ =\ (u_1,u_2,u_3,u_4,\ol{u}_1,\ol{u}_2,\ol{u}_3,\ol{u}_4)\in\CC^4\times\CC^4\,,
\]
where
\[
  u_0 = -(u_1+u_2+u_4)\,,\ 
  \ol{u}_0 = -(\ol{u}_1+\ol{u}_2+\ol{u}_4)\,,\ 
  v_4 = -(u_1+u_3+1)\,,\
  \ol{v}_4 = -(\ol{u}_1+\ol{u}_3+1)\,.
\]
A mechanism is \defcolor{physically meaningful}
if $\ol\bfu = {\rm conjugate}(\bfu)$.

For $(\btheta,\ol\btheta)\in\CC^4\times\CC^4$, the poses corresponding to link lengths
$(\bfu,\ol\bfu)$ satisfy
\[
  G(\bfu,\ol\bfu,\btheta,\ol\btheta)\ :=\ 
     \left(\begin{array}{c}
       \theta_1 \ol{\theta}_1 - 1 \\
       \theta_2 \ol{\theta}_2 - 1 \\
       \theta_3 \ol{\theta}_3 - 1 \\
       \theta_4 \ol{\theta}_4 - 1 \\
       u_0+u_1\theta_1+u_2\theta_2+u_4\theta_4 \\
       \ol{u}_0+\ol{u}_1\ol{\theta}_1+\ol{u}_2\ol{\theta}_2+\ol{u}_4\ol{\theta}_4 \\ 
       1+u_1\theta_1+u_3\theta_3+v_4\theta_4 \\
       1+\ol{u}_1\ol{\theta}_1+\ol{u}_3\ol{\theta}_3+\ol{v}_4\ol{\theta}_4
     \end{array}\right)\ =\ 0\,.
\]
The six configurations of a general planar
pentad correspond with the six points
in the witness point set
$W_\bfe = \calV(F)\cap\calV(\bL^\bfe)$ 
for $\bfe = (4,4,0,0)$.

The only family of nondegenerate 
exceptional planar pentads are the double-parallel\-ogram linkages~\cite{SWAlgKinematics}, namely
\[
     \mathcal{U}\ :=\ \{(\bfu,\ol\bfu)\mid u_1+u_2=u_1+u_3=\ol{u}_1+\ol{u}_2=
     \ol{u}_1+\ol{u}_3=0\}\ \subset\ \CC^4\times\CC^4\,.
\]
We aim to compute $\mathcal{U}$
directly from $G$ by using fiber products.
Since $\mathcal{U}$ has codimension four,
Corollary~2.14 of \cite{SWFiberProducts} shows
that $\mathcal{U}$ will correspond with
an irreducible component of the fourth
fiber product system, namely
\[
  F(\bfu,\ol\bfu,
     \btheta_1,\ol\btheta_1,
     \btheta_2,\ol\btheta_2,
     \btheta_3,\ol\btheta_3,
     \btheta_4,\ol\btheta_4)
     \ =\ \left(\begin{array}{c}
              G(\bfu,\ol\bfu, \btheta_1,\ol\btheta_1) \\
              G(\bfu,\ol\bfu,\btheta_2,\ol\btheta_2) \\
              G(\bfu,\ol\bfu,\btheta_3,\ol\btheta_3) \\
              G(\bfu,\ol\bfu,\btheta_4,\ol\btheta_4)
        \end{array}\right)\ =\ 0\,.
\]
The irreducible component $X\subset\calV(F)\subset(\CC^4)^{10}$
corresponding to $\mathcal{U}$
is called a \defcolor{main component} 
in \cite{SWFiberProducts}.  
In fact, $X$ is a Cartesian product of $\mathcal{U}$ with four copies of 
$${\{(\alpha,\alpha,\alpha,1,\, \alpha^{-1},\alpha^{-1},\alpha^{-1},1)
\in \mathbb{C}^4 \times \mathbb{C}^4
\mid\alpha \in \mathbb{C}^*\}}.$$
This corresponds with rotating $\bigtriangleup ABC$ about a fixed $\bigtriangleup DEF$.

A necessary condition for locating 
such an exceptional component 
of codimension four consisting
of nondegenerate and physically meaningful
linkages with one dimension of motion
is that there exists 
$\bfe\in\Dim(\calV(F))$ 
such that $e_1=e_2=2$ and $e_3+e_4=e_5+e_6=e_7+e_8=e_9+e_{10}=1$.
A sufficient condition for such a component to exist is that one of the isolated points in $\calV(F)\cap \mathcal{V}(L^\bfe)$ is a general point of $X$. 
To that end, we first compute 
the $\bfe$-witness point set $W_\bfe$
for $\calV(F)$ where $\bfe=(2,2,1,0,1,0,1,0,1,0)$
resulting in 14,828 isolated nonsingular
points in $\calV(F)\cap \mathcal{V}(L^\bfe)$.
The results of using 
Algorithm~\ref{A:local-dimension}
to compute the the local dimension at each 
of these 14,828 witness points along with the
local fiber dimensions obtained by
fixing the $(\bfu,\ol\bfu)$ coordinates
and coarsening the fiber to the natural $(\CC^8)^4$
are summarized in the following.
 \[
   \begin{array}{c|c|c}  
     \hbox{$\#$~points~in~$W_\bfe$} & 
\hbox{local dimension $\Dim_x(F)$} & 
\hbox{local fiber dimension over $(\bfu,\ol\bfu)$}
 \\\hline
\hbox{14,144} & (4,4,4,4,4,4,4,4,4,4) & (0,0,0,0) \\
\hbox{678} &   (2,2,3,3,3,3,3,3,3,3) & (1,1,1,1) \\
\hbox{6} &     (2,2,1,1,1,1,1,1,1,1) & (1,1,1,1)
\end{array}
\]
The first collection of 14,144 
witness points correspond with rigid planar
pentads.  The second collection 
of 678 witness points correspond with degenerate
planar pentads.  The final collection
of 6 witness points correspond
with five degenerate planar pentads
arising from one of the five
nonconstant edges of $\bigtriangleup ABC$ or $\bigtriangleup DEF$ being 
zero (recall \hbox{$v_0 = \ol{v}_0=1$}).
The other witness point in this collection
is the unique point in $W_\bfe$ 
that is a general point on $X$ 
thereby confirming
$(2,2,1,0,1,0,1,0,1,0) \in \Dim(\calV(F))$
and the existence of nondegenerate exceptional
planar pentads.
On the other hand, working without the multihomogeneous structure,  leads to approximately $10^8$ witness points which we can't robustly compute.

\subsection*{Acknowledgements}
We thank Tim Duff for his helpful comments on the paper. 


\providecommand{\bysame}{\leavevmode\hbox to3em{\hrulefill}\thinspace}
\providecommand{\MR}{\relax\ifhmode\unskip\space\fi MR }
\providecommand{\MRhref}[2]{%
  \href{http://www.ams.org/mathscinet-getitem?mr=#1}{#2}
}
\providecommand{\href}[2]{#2}

\end{document}